\documentclass[11pt,a4paper]{article}
\usepackage{bbm}

\usepackage[leqno]{amsmath}
\usepackage{amsfonts}
\usepackage{graphicx}
\usepackage{amsmath}
\usepackage{amssymb}
\usepackage{latexsym}
\usepackage{amsmath, amsfonts,amssymb, amsthm, euscript,makeidx,color,mathrsfs}
\usepackage{enumerate}

\usepackage[colorlinks,linkcolor=blue,anchorcolor=green,citecolor=red]{hyperref}

\def\sqr#1#2{{\vcenter{\vbox{\hrule height.#2pt
              \hbox{\vrule width.#2pt height#1pt \kern#1pt \vrule width.#2pt}
              \hrule height.#2pt}}}}
\def\mf{\mathcal{F}}
\def\me{\mathbb{E}}
\def\bal{\begin{aligned}}
\def\eal{\end{aligned}}

\def\5n{\negthinspace \negthinspace \negthinspace \negthinspace \negthinspace }
\def\4n{\negthinspace \negthinspace \negthinspace \negthinspace }
\def\3n{\negthinspace \negthinspace \negthinspace }
\def\2n{\negthinspace \negthinspace }
\def\1n{\negthinspace }

\def\dbE{\mathbb{E}}
\def\dbF{\mathbb{F}}

\def\dbN{\mathbb{N}}

\def\dbR{\mathbb{R}}

\def\={\buildrel \triangle \over =}

\def\ds{\displaystyle}

\def\a{\alpha}

\def\g{\gamma}

\def\e{\varepsilon}

\def\si{\sigma}
\def\t{\tau}
\def\f{\varphi}
\def\th{\theta}
\def\o{\omega}

\def\D{\Delta}

\def\O{\Omega}
\def\no{\noindent}

\def\ms{\medskip}

\def\q{\quad}
\def\qq{\qquad}

\def\lan{\langle}
\def\ran{\rangle}

\def\cd{\cdot}

\def\({\Big (}
\def\){\Big )}
\def\[{\Big[}
\def\]{\Big]}

\def\bde{\begin{definition}\label}
\def\ede{\end{definition}}
\def\be{\begin{equation}}
\def\bel{\begin{equation}\label}
\def\ee{\end{equation}}
\def\bt{\begin{theorem}\label}
\def\et{\end{theorem}}
\def\bc{\begin{corollary}\label}
\def\ec{\end{corollary}}
\def\bl{\begin{lemma}\label}
\def\el{\end{lemma}}
\def\bp{\begin{proposition}\label}
\def\ep{\end{proposition}}
\def\bas{\begin{assumption}\label}
\def\eas{\end{assumption}}
\def\br{\begin{remark}\label}
\def\er{\end{remark}}
\def\bex{\begin{example}\label}
\def\ex{\end{example}}
\def\ba{\begin{array}}
\def\ea{\end{array}}
\def\ed{\end{document}}

\def\square#1{\vbox{\hrule\hbox{\vrule height#1%
     \kern#1\vrule}\hrule}}
\def\rectangle#1#2{\vbox{\hrule\hbox{\vrule height#1%
     \kern#2\vrule}\hrule}}

\font\tenbb=msbm10 \font\sevenbb=msbm7 \font\fivebb=msbm5

\newfam\bbfam
\scriptscriptfont\bbfam=\fivebb \textfont\bbfam=\tenbb
\scriptfont\bbfam=\sevenbb

\newtheorem{theorem}{\hskip 1.3em Theorem}[section]
\newtheorem{definition}[theorem]{\hskip 1.3em Definition}
\newtheorem{proposition}[theorem]{\hskip 1.3em Proposition}
\newtheorem{corollary}[theorem]{\hskip 1.3em Corollary}
\newtheorem{lemma}[theorem]{\hskip 1.3em Lemma}
\newtheorem{remark}[theorem]{\hskip 1.3em Remark}
\newtheorem{example}[theorem]{\hskip 1.3em Example}

\newtheorem{assumption}[theorem]{\hskip 1.3em Assumption}

\makeatletter
   
   \@addtoreset{equation}{section}
\makeatother

\begin{document}

\title{\bf A Numerical Scheme for BSVIEs\thanks{This work is supported in part by
the National Natural Science Foundation of China (11526167), the Fundamental Research Funds for the Central
Universities (SWU113038, XDJK2014C076),  the Natural Science
Foundation of CQCSTC (2015jcyjA00017).}}

\author{Yanqing Wang\footnote{School of Mathematics and Statistics, Southwest University, Chongqing 400715, China; email: {\tt yqwang@amss.ac.cn}}}

\maketitle

\begin{abstract}
In this paper, we consider the Euler method for backward stochastic Volterra integral equations.
First, we approximate the original equation by a family of backward stochastic equations (BSDEs, for short).
 Then we solve the BSDEs by the Euler method. Finally, by virtue of the numerical solutions to BSDEs, we get the numerical
 solution to original equation and obtain the global $1/2$ order convergence speed in $L^2$ norm. 
\end{abstract}

\ms

\no\bf Keywords: \rm Backward stochastic Volterra integral equation, the Euler method, backward stochastic 
differential equation, Malliavin analysis.

\ms

\no\bf AMS subject classification: \rm 60H20, 65C30

\maketitle

\section{Introduction}

Throughout this paper, we
let $T\in(0,+\infty)$, $(\O,\mf,\mathbb{F},P)$ be a complete probability space and
$\mathbb{F}=\{\mf_t,t\in [0,T]\}$ be the natural filtration generalized by a 1-dimensional Wiener process  $\{W(t):t\in [0,T]\}$ satisfying the usual conditions.
The purpose of this work is to present a numerical scheme for solving the following backward stochastic Volterra integral equation (BSVIE, for short):
\begin{equation}\label{bsvie1}
Y(t)=g(t,x(T))+\int_t^T f(t,s,x(s),Y(s),Z(t,s))ds-\int_t^T Z(t,s)dW(s),\,\,t\in [0,T],
\end{equation}
where $\ds f: \D^c\times \dbR^d \times \dbR^n\times \dbR^n\rightarrow \dbR^n$, $\ds g: [0,T]\times \dbR^d\rightarrow \dbR^n$ are given maps with $\ds \D^c=\{(t,s)\in [0,T]^2:t<s\}$, and 
 $x(\cdot)$ satisfies the following stochastic Volterra integral equation (SVIE, for short):
\begin{equation}\label{svie}
x(t)=\varphi(t)+\int_0^t b(t,s,x(s))ds+\int_0^t \si(t, s,x(s))dW(s),\,\,t\in [0,T].
\end{equation}
Here $\ds \varphi: [0,T]\times \O\rightarrow \dbR^d,\, b,\,\si: [0,T]^2\times \dbR^d\rightarrow \dbR^d. $ 

\ms

BSVIEs are natural and nontrivial extensions of  backward stochastic differential equations 
(BSDEs, for short), and  the general  BSVIEs can not be reduced to BSDEs (see \cite{Yong08}).
The main feature of SVIEs/BSVIEs is that these equations contain memories, which is closer to reality.
We refer to \cite{Berger-Mizel80}, \cite{Ito79} for the pioneering work on SVIEs. 
Nonlinear BSVIEs was 
first introduced in 2002 (\cite{Lin02}). Later, Yong (\cite{Yong08}) studied the well-posedness of 
solutions to generalized BSVIEs. Thereafter,
BSVIEs turned out to be an extremely useful tool in the study of stochastic control problems 
for SVIEs, time-inconsistent stochastic
differential utility and risk management (see, e.g., \cite{Chen-Yong07, Yong07}).

\ms

Generally, it is impossible to obtain the true solutions to BSDEs/BSIVEs. Hence, the study of numerical solutions becomes necessary and interesting. In recent period, the study of numerical solutions to stochastic differential
equations (SDEs, for short) becomes an active topic.
So far, the following numerical schemes for BSDEs have been presented: the four step scheme, the Euler
method, the random walk approach, the Wiener chaos expansion method, the finite transposition method and
so on (see, e.g., \cite{Bender-Denk07,Bouchard-Touzi04,Briand-Labart14,Douglas-Ma-Protter96,Ma-Protter02,Milstein-Tretyakov06,Wang-Zhang11,ZhangJF04}). But for BSVIE, the numerical method is quiet limited. Here  we mention \cite{Bender-Pokalyuk13}.
In \cite{Bender-Pokalyuk13}, the numerical method for the following BSVIE is considered:
\begin{equation}\label{bsvie00}
\begin{aligned}
Y(t)=g(t,W)+\int_t^Tf(s,Y(s))ds-\int_t^TZ(t,s)dW(s),\,t\in [0,T],
\end{aligned}
\end{equation}
which is approximated by a family of discrete BSVIEs driven by a binary random walk with solutions $(Y^{(n)},Z^{(n)})$. Under suitable conditions, $Y^{(n)}$ converges weakly to $Y$ in the Skorokhod topology.
That result relies on a representation for BSVIEs by systems of quasilinear PDEs of parabolic type.
 
\ms

In this paper, we employ the Euler method to present the numerical solution to BSVIE \eqref{bsvie1}. 
To be specific, suppose a partition  $\pi:0=t_0<t_1<\cdots<t_N=T$ of $[0,T]$ with the mesh
 size $\displaystyle|\pi|=\max_{0\leq i\leq N} |t_{i+1}-t_i|$. Then we denote $\D_i=t_{i+1}-t_i$ and $\D_i
 W=W(t_{i+1})-W(t_i)$, for $i = 0, 1, \cdots, N-1$.  
 
 \ms
 
 For  $0\leq k\leq N-1$, we present the Euler method for BSVIE \eqref{bsvie1} as follows:
\begin{equation}\label{ass01}
\left\{
\begin{aligned}
&Y^{k,\pi}(t_N)=g(t_k,x^\pi(T)),\\
&Y^{k,\pi}(t_l)=\me\Big(Y^{k,\pi}(t_{l+1})
        +f(t_k,t_l,x^\pi(t_l),Y^{l,\pi}(t_{l+1}),Z^{k,\pi}(t_{l}))\D_l\big|\mf_{t_l}\Big),\\
&Z^{k,\pi}(t_l)=\me\Big(\frac{\D_lW}{\D_l}\big(Y^{k,\pi}(t_{l+1})
               +f(t_k,t_l,x^\pi(t_l),Y^{l,\pi}(t_{l+1}),Z^{k,\pi}(t_{l}))\big)\big|\mf_{t_l}\Big),\\
               &\qq\qq\qq\qq\qq\qq\qq\qq\qq\qq\, k\leq l\leq N-1.\\
\end{aligned}
\right.
\end{equation}
Here $x^\pi(\cdot)$ is the numerical solution to SVIE \eqref{svie} stated as 
\begin{equation}\label{svieeuler0}
\left\{
\begin{aligned}
x^\pi(0)=&x^\pi(t_0)=\varphi(0),\\
x^\pi(t_{i+1})=&\varphi(t_{i+1})+\sum_{k=0}^i \Big[b(t_{i+1},t_k,x^\pi(t_k))\D_k
                      +\sigma(t_{i+1},t_k,x^\pi(t_k))\D_kW\Big], \\
&\qq\qq\qq\qq\qq\qq i=0,1,\cdots, N-1.\\
\end{aligned}
\right.
\end{equation}
Under suitable conditions on $f,\,g,\,\varphi,\,b$ and $\si$ (assumptions (A1)--(A4) below), 
in the cases: (I) $f=f(t,s,x,y)$; (II) $f=f(t,s,x,z)$,
we can prove that (Theorem \ref{convergence})
\begin{equation}\label{conver0}
\begin{aligned}
\max_{0\leq k\leq N}\me|Y(t_k)-Y^{k,\pi}(t_k)|^2
         +\sum_{k=0}^{N-1}\me\int_{t_k}^{t_{k+1}}\int_t^T|Z(t,s)-Z^{k,\pi}(\t(s))|^2ds
\leq K|\pi|,
\end{aligned}
\end{equation}
where $\t(\cdot) $ is a map  on $[0,T)$ defined by $\t(s)=t_i,\, s\in [t_i, t_{i+1}),\,i=0,1,\cdots,N-1$ and $K$ is a constant.

 \ms
 
 The rest of the paper is organized as follows: In Section 2, we review some of the standard results on SDEs and BSDEs,   introduce our general setting and show the well-posedness of SVIE \eqref{svie} and BSVIE \eqref{bsvie1}. In Section 3, we present the Euler method to obtain the numerical solution to SVIE \eqref{svie} and get the convergence speed. 
In Section 4,  we adopt the  Euler method for BSVIE \eqref{bsvie1}, and the convergence and error analysis are also provided. A numerical example is presented in Section 5.

\ms
 
\section{Preliminaries}
Recall that $\dbR^n$ is the $n$-dimensional Euclidean
 space with the standard Euclidean norm $|\cd|$ induced by the
standard Euclidean inner product $\lan\cd\,,\cd\ran$. Hereafter, the
superscript $^\top$ denotes the transpose of a vector or a matrix. 
We
now introduce some spaces: for $p,q\geq 1$,
\begin{itemize}
\item $L^p_{\mf_T}(\O;\dbR^n)$ is the space of all $\mf_T$-measurable random variances $\xi$ valued in $\dbR^n$ such that
$$\|\xi\|_{L^p_{\mf_T}(\O;\dbR^n)}=\big(\dbE|\xi|^p\big)^{1\over p}<\infty.$$
\item $L^p_\dbF(\O;L^q(0,T;\dbR^n))$ is the space of all $\dbF$-progressively measurable processes $\f(\cd)$ valued in $\dbR^n$ such that
$$\|\f(\cd)\|_{L^p_\dbF(\O;L^q(0,T;H))}=\[\dbE\Big(\int_0^T|\f(t)|^qdt\Big)^{p\over q}\]^{1\over p}<\infty.$$
When $p=q$, we write $\ds L^p_\dbF(\O\times(0,T);\dbR^n)$ for simplicity.
\item $\mathbb{L}_a^{1,2}(\dbR^n)$ is the space of all $\dbF$-progressively measurable processes
$u(\cdot)$ valued in $\dbR^n$ satisfying
\begin{description}
\item[(i)] For almost all $t\in[0,T]$, $u(t)\in \mathbb{D}^{1,2}(\dbR^n)$;
\item[(ii)] $\ds \me\Big(\int_0^T |u(t)|^2 dt+\int_0^T\int_0^T|D_{\theta}u(t)|^2d\theta  dt\Big)<\infty.$
\end{description}
\end{itemize}

\ms

The following lemma collects some standard results in SDE and BSDE literature. We only list them.
\begin{lemma}\label{aaaa}
Suppose that $b_0,\si_0:\O\times [0,T]\times \dbR^d\rightarrow \dbR^d$ and 
$f_0:\O\times [0,T]\times \dbR^n\times \dbR^n\rightarrow \dbR^n$ are $\dbF$-adapted random fields, satisfying:\\
(a) they are uniformly Lipschitz continuous with respect to $x\in \dbR^d$, $y\in \dbR^n$ and $z\in \dbR^n$,\\
(b) $b_0(\cdot,0)$, $\si_0(\cdot,0)\in L^2_{\dbF}(\O\times (0,T);\dbR^d)$, $f_0(\cdot,0,0)\in L^2_{\dbF}(\O\times (0,T);\dbR^n)$.\\
For any $x\in \dbR^d$ and $\xi\in L^2_{\mf_T}(\O;\dbR^n)$, $X(\cdot)$ is the solution to the following SDE:
$$X(t)=x+\int_0^t b_0(s,X(s))ds+\int_0^t \si_0(s,X(s))dW(s),\,\,t\in [0,T],$$
and $(Y(\cdot),Z(\cdot))$ solves the BSDE:
$$Y(t)=\xi+\int_t^Tf_0(s,Y(s),Z(s))ds-\int_t^TZ(s)dW(s),\,\,t\in[0,T].$$
Then, for any $p\geq 2$, we have the following estimates:
$$\me\big(\sup_{0\leq t\leq T}|X(t)|^p\big)\leq C\bigg\{|x|^p+\me\Big(\int_0^T|b_0(t,0)|dt\Big)^p+\me\Big(\int_0^T|\si_0(t,0)|^2dt\Big)^{p/2}\bigg\},$$
$$\me\big(\sup_{0\leq t\leq T}|Y(t)|^p\big)+\me\Big(\int_0^T|Z(s)|^2ds\Big)^{p/2}\leq C\bigg\{\me|\xi|^p+\me\Big(\int_0^T|f_0(s,0,0)|ds\Big)^p\bigg\},$$
where $C$ is a constant.
\end{lemma}

Throughout the paper, we will make use of  the following assumptions.\\
\textbf{(A1)}  $\ds f: [0,T]^2\times \dbR^d\times \dbR^n\times \dbR^n \rightarrow \dbR^n$, and there exists a constant $L$ such that
\begin{equation}\label{f1}
\ba{cc}
&\ds |f(t_1,s_1,x,y,z)-f(t_2,s_2,x, y,z)|
\leq L(|t_1-t_2|^{1/2}+|s_1-s_2|^{1/2}),\\
&\ds \qq\qq\qq\qq\qq\qq\,\,s_1,s_2\in (\max\{t_1,t_2\},T],\,x\in \dbR^d,\,y,\,z\in \dbR^n,\\
&\ds |f(\cdot,\cdot,0,0,0)|\leq L,\\
\ea
\end{equation}
and $f$ has continuous and uniformly bounded first and second partial derivatives with 
respect to $x$, $y$ and $z$ (boundary is $L$). \\
\textbf{(A2)} $\ds g: [0,T]\times \dbR^d \rightarrow \dbR^n$, and there exists a constant $L$ such that
\begin{equation}\label{g1}
\ba{c}
\ds |g(t_1,x)-g(t_2,x)|\leq L|t_1-t_2|^{1/2},\,\,t_1,t_2\in [0,T],\,\,x\in\dbR^d,\\
\ds |g(\cdot,0)|\leq L,
\ea
\end{equation}
and $g$ has continuous and uniformly bounded first and second partial derivatives 
with respect to $x$ (boundary is L).\\
\textbf{(A3)}  $\ds b,\sigma: [0,T]^2\times \dbR^d \rightarrow \dbR^d$, and there exists a constant $L$ such that\begin{equation}\label{b1}
\ba{cc}
&\ds |b(t_1,s_1,x)-b(t_2,s_2,x)|+|\sigma(t_1,s_1,x)-\sigma(t_2,s_2,x)|
\leq
     L(|t_1-t_2|^{1/2}+|s_1-s_2|^{1/2}), \\
     &\ds \qq\qq\qq\qq\qq\,\,t_1,\,t_2,\,s_1,\,s_2\in [0,T],\,x\in \dbR^d,\\
&\ds  |b(\cdot,\cdot,0)|+|\sigma(\cdot,\cdot,0)|\leq L,\\
\ea
\end{equation}
and $b,\,\si$ has continuous and uniformly bounded first and second partial derivatives with 
respect to $x$ (boundary is $L$). \\
\textbf{(A4)} $\varphi(\cdot)$ is $\dbF$-adapted continuous process and there exists a constant $p_0>2$ and $L$ such that
\begin{equation}\label{varphi1}
\ba{c}
\ds \me|\varphi(t)-\varphi(s)|^2\leq L|t-s|,\,\,t,s\in [0,T],\\
\ds \me|D_{\th_1}\varphi(t)-D_{\th_2}\varphi(t)|^2\leq L|\th_1-\th_2|,\, 0\leq \th_1,\th_2\leq t\leq T,\\
\ds \sup_{0\leq\th_1,\th_2\leq t\leq T}\me\[|\varphi(t)|^{2p_0}+|D_{\th_1}\varphi(t)|^{2p_0}
        +|D_{\th_1}D_{\th_2}\varphi(t)|^{p_0}\]\leq L^{2p_0}.
\ea
\end{equation}

In what follows, $K$ and $C$ are positive constants, depending 
only on $L$ and $T$, and may be different from line to line.

\subsection {Regularity of $x(\cdot)$}
In this part, we review the wellposedness of SVIE \eqref{svie}.
Under assumptions {\rm(A3)--(A4)}, the wellposedness of SVIEs can be proved by a routine successive approximation argument (\cite{Ito79}). 
The following properties on $x(\cdot)$ are need later.

\begin{lemma}\label{sviex}
Under assumptions {\rm(A3)--(A4)}, for any  $0\leq t_0\leq t\leq T$, $0\leq \th_1,\th_2\leq t$, it holds that
\begin{equation}\label{regularityx}
\ba{c}
\ds \sup_{0\leq\th_1,\th_2\leq t\leq T}\me\Big(|x(t)|^{2p_0}+|D_{\th_1}x(t)|^{2p_0}+|D_{\th_1}D_{\th_2}x(t)|^{p_0}\Big)<K,\\
\ds \me|x(t)-x(t_0)|^2\leq K(t-t_0),\\
\ds \sup_{\th_1,\th_2\leq t\leq T}\me|D_{\th_1}x(t)-D_{\th_2}x(t)|^2\leq K|\th_1-\th_2|,
\ea
\end{equation}
where $K$ is a constant depending only on $p_0$, $L$ and $T$.
\end{lemma}

\begin{proof}
Suppose that $0\leq t_0\leq t\leq T$. Then by SVIE \eqref{svie}, one obtains
\begin{equation}\label{estix1}
\begin{aligned}
\me|x(t)|^{2p_0}
\leq& 3^{2p_0-1}\me|\varphi(t)|^{2p_0}+(3T)^{2p_0-1}\me\int_0^t|b(t,s,x(s))|^{2p_0}ds\\
      &  +3^{2p_0-1}T^{\frac {2p_0} 2-1}\me\int_0^t|\si(t,s,x(s))|^{2p_0}ds\\
 \leq & 3^{2p_0-1}L^{2p_0}+(6T)^{2p_0-1}\me\int_0^t|b(t,s,x(s))-b(t,s,0)|^{2p_0}+|b(t,s,0)|^{2p_0}ds\\
             &+6^{2p_0-1}T^{\frac {2p_0} 2-1}\me\int_0^t|\si(t,s,x(s))-\si(t,s,0)|^{2p_0}+|\si(t,s,0)|^{2p_0}ds\\
\leq & K+K\me\int_0^t|x(s)|^{2p_0}ds.\\
\end{aligned}
\end{equation}
Consequently, by virtue of Gronwall's inequality, we can get $\ds \sup_{0\leq t\leq T}\me|x(t)|^{2p_0}<K$. 
Also by the routine successive approximation argument, $D_\th x(\cdot)$, the Malliavin derivative of 
$x(\cdot)$, satisfies the following SVIE: 
for any $0\leq \th \leq t \leq T$
\begin{equation*}\label{estix5}
\begin{aligned}
D_\th x(t)=&D_\th\varphi(t)+\sigma(t,\th,x(\th))+\int_\th^t b_x(t,s,x(s))D_\th x(s)ds
        +\int_\th^t \sigma_x(t,s,x(s))D_\th x(s)dW(s).
\end{aligned}
\end{equation*}
Similarly, we also can obtain
$\ds \sup_{0\leq t\leq T,\, 0\leq\th_1,\th_2\leq t}\me\Big(|D_{\th_1}x(t)|^{2p_0}+|D_{\th_1}D_{\th_2}x(t)|^{p_0}\Big)<K$, 
which is the first inequality of  \eqref{regularityx}.

\ms

Now, making use of the first inequality of \eqref{regularityx},  with the similar estimate to 
that of \eqref{estix1}, we can obtain
\begin{equation*}\label{estix2}
\begin{aligned}
\me|x(t)-x(t_0)|^2
\leq &5\me|\varphi(t)-\varphi(t_0)|^2+5T\me\int_0^{t_0}|b(t,s,x(s))-b(t_0,s,x(s))|^2ds\\
  &+5\me\int_0^{t_0}|\si(t,s,x(s))-\si(t_0,s,x(s))|^2ds\\
  &+5\me\Big|\int_{t_0}^t b(t,s,x(s))ds\Big|^2+5\me\int_{t_0}^t|\si(t,s,x(s))|^2ds\\
\leq & K|t-t_0|+K\me\int_{t_0}^t|x(s)|^2ds\\
\leq & K|t-t_0|,\\
\end{aligned}
\end{equation*}
which is the second inequality of \eqref{regularityx}.

\ms

For the third inequality of \eqref{regularityx}, suppose that $\th_2\leq \th_1\leq t$.
Since
\begin{equation*}\label{april9}
\begin{aligned}
D_{\th_1}x(t)-D_{\th_2}x(t)
=&\big(D_{\th_1}\f(t)-D_{\th_2}\f(t)\big)+\big(\si(t,\th_1,x(\th_1))-\si(t,\th_2,x(\th_2))\big)\\
   &+\int_{\th_1}^t b_x\big(D_{\th_1}x(t)-D_{\th_2}x(t)\big)ds
         +\int_{\th_1}^t \si_x\big(D_{\th_1}x(t)-D_{\th_2}x(t)\big)dW(s)\\
   &-\int_{\th_2}^{\th_1} b_xD_{\th_2}x(t)ds-\int_{\th_2}^{\th_1} \si_xD_{\th_2}x(t)dW(s),\\
\end{aligned}
\end{equation*}
it easy to calculate that
\begin{equation*}\label{april10}
\begin{aligned}
&\me|D_{\th_1}x(t)-D_{\th_2}x(t)|^2\\
=&6\me|D_{\th_1}\f(t)-D_{\th_2}\f(t)|^2+6L^2\big(|\th_1-\th_2|+\me|x(\th_1)-x(\th_2)|^2\big)\\
   &+6(T+1)L^2\me\int_{\th_1}^t |D_{\th_1}x(s)-D_{\th_2}x(s)|^2ds
      +6(T+1)L^2\me\int_{\th_2}^{\th_1} |D_{\th_2}x(s)|^2ds\\
\leq & K|\th_2-\th_1|+6(T+1)L^2\me\int_{\th_1}^t |D_{\th_1}x(s)-D_{\th_2}x(s)|^2ds.\\
\end{aligned}
\end{equation*}
Hence, by Gronwall's inequality, we have
\begin{equation*}\label{april11}
\begin{aligned}
\sup_{\th_1,\th_2\leq t\leq T}\me|D_{\th_1}x(t)-D_{\th_2}x(t)|^2\leq K|\th_1-\th_2|,
\end{aligned}
\end{equation*}
 completing the proof.
\end{proof}

\subsection{Regularity of $(Y(\cdot),Z(\cdot,\cdot))$}
The following result on wellpossedness of BSVIE \eqref{bsvie1} comes from \cite[Theorem 3.7 and 4.1]{Yong08}.
\begin{theorem}\label{yong08}
Under assumptions {\rm(A1)--(A4)}, BSVIE \eqref{bsvie1} admits a unique solution $(Y(\cdot),Z(\cdot,\cdot))$.
Moreover, the following estimates holde:
\begin{equation}\label{yong081}
\begin{aligned}
&\me\int_S^T|Y(t)|^2dt+\me\int_S^T\int_t^T|Z(t,s)|^2dsdt\\
 & \leq C\bigg\{\me\int_S^T|g(t,x(T))|^2dt+\me\int_S^T\bigg(\int_t^T|f(t,s,0,0, 0)|ds\bigg)^2dt\bigg\},
         \,\,\mbox{for any }S\in [0,T],\\
\end{aligned}
\end{equation}
\begin{equation}\label{yong082}
\begin{aligned}
&\sum_{i=1}^n\me\bigg\{\int_S^T|D_r^iY(t)|^2dt+\int_S^T\int_t^T|D_r^iZ(t,s)|^2dsdt\bigg\}\\
  \leq& C\me\bigg\{\int_S^T|g(t,x(T))|^2dt+\sum_{i=1}^n\int_S^T|D_r^ig(t,x(T))|^2dt\\
    &\qq +\int_S^T\bigg(\int_t^T|f(t,s,0,0,0)|ds\bigg)^2dt\bigg\},
         \,\,\mbox{for any }r,\,S\in [0,T].\\
\end{aligned}
\end{equation}
Morevoer, $(D_r^iY(\cdot),D_r^iZ(\cdot,\cdot))$ is the adapted solution to the following 
BSVIE:
\begin{equation}\label{yong083}
\begin{aligned}
D_r^iY(t)
=&D_r^ig(t,x(T))+\int_t^T \Big(g_x(t,s,x(s),Y(s),Z(t,s))D_r^ix(s)\\
   &+g_y(t,s,x(s),Y(s),Z(t,s))D_r^iY(s)\\
          &+g_z(t,s,x(s),Y(s),Z(t,s))D_r^iZ(t,s)\Big)ds\\
          &-\int_r^TD_r^iZ(t,s)dW(s),\,\, t\in [r,T].
\end{aligned}
\end{equation}
In addition, for any $0\leq t<u\leq T$, $1\leq i\leq n$, 
\begin{equation}\label{yong084}
\begin{aligned}
Z_i(t,u)=&D_u^ig(t,x(T))+\int_u^T\Big(f_x(t,s,x(s),Y(s),Z(t,s))D_u^ix(s)\\
           &+f_y(t,s,x(s),Y(s),Z(t,s))D_u^iY(s)\\
            &+f_z(t,s,x(s),Y(s),Z(t,s))D_u^iZ(t,s)\Big)ds\\
            &-\int_u^TD_u^iZ(t,s)dW(s).
\end{aligned}
\end{equation}
\end{theorem}

\ms

The following result is used to deduce the convergence speed in the Euler method for BSVIE \eqref{bsvie1}.
\begin{lemma}\label{estimateyz}
Under assumptions {\rm(A1)--(A4)}, for any $t,\,t_0\in [0,T]$, it holds that
\begin{equation}\label{estiyz5}
\begin{aligned}
\me|Y(t)-Y(t_0)|^2+\me\int_{t\vee t_0}^T|Z(t,s)-Z(t_0,s)|^2ds\leq C|t-t_0|,
\end{aligned}
\end{equation}
where $C$ is a constant.
\end{lemma}

\begin{proof}

Suppose that $t_0<t$. By \cite[Corrolary 3.6]{Yong08}, under assumptions (A1)--(A4), we have
\begin{equation}\label{estiyz1}
\begin{aligned}
&\me|Y(t)-Y(t_0)|^2+\me\int_t^T|Z(t,s)-Z(t_0,s)|^2ds\\
\leq & C\bigg\{\me|g(t,x(T))-g(t_0,x(T))|^2+\me\Big(\int_{t_0}^t |f(t_0,s,x(s),Y(s),Z(t_0,s))|ds\Big)^2\\
       &\qq+ \me\Big(\int_t^T|f(t,s,x(s),Y(s),Z(t,s))-f(t_0,s,x(s),Y(s),Z(t,s))|ds\Big)^2\\
       &\qq+\me\int_{t_0}^t|Z(t_0,s)|^2ds\bigg\}\\
\leq & C|t-t_0|+C\me\int_{t_0}^t\big(|Y(s)|^2+|Z(t_0,s)|^2\big)ds.\\
\end{aligned}
\end{equation}
For $\ds\me|Y(\cdot)|^2$, also by  \cite[Corrolary 3.6]{Yong08}, one has
\begin{equation*}\label{estiyz001}
\begin{aligned}
\me|Y(t)|^2+\me\int_t^T|Z(t,s)|^2ds
\leq & C\bigg\{\me|g(t,x(T))|^2+\me\Big(\int_t^T|f(t,s,x(s),Y(s),0)|ds\Big)^2\bigg\}\\
\leq & C+C\me\int_{t}^T |Y(s)|^2 ds.\\
\end{aligned}
\end{equation*}
By Gronwall's inequality, one gets that $\ds\sup_{0\leq t\leq T}\me|Y(t)|^2\leq C$. Thus 
\begin{equation}\label{estiyz002}
\begin{aligned}
\me\int_{t_0}^t|Y(s)|^2ds\leq C|t-t_0|.
\end{aligned}
\end{equation}
Setting $t=t_0$ in \eqref{yong084}, by \cite[Corollary 3.6]{Yong08}, Lemma \ref{sviex} and \eqref{yong082},
we can obtian
\begin{equation}\label{estiyz003}
\begin{aligned}
 &\me|Z(t_0,u)|^2+\me\int_u^T|D_uZ(t_0,s)|^2ds\\
\leq & C\bigg\{\me|D_ug(t_0,x(T))|^2+\me\Big(\int_u^T\big(|f_x(t_0,s,x(s),Y(s),Z(t_0,x))D_ux(s)|\\
       &\qq+|f_y(t_0,s,x(s),Y(s),Z(t_0,x))D_uY(s)|\big)ds\Big)^2\bigg\}\\
\leq & C\bigg\{\me|D_ux(T)|^2+\me\int_u^T\big(|D_ux(s)|^2+|D_uY(s)|^2\big)ds\bigg\}\\
<&\infty.
\end{aligned}
\end{equation}
Now, \eqref{estiyz1}, together with \eqref{estiyz002} and \eqref{estiyz003}, yields that 
\begin{equation*}\label{estiyzt}
\begin{aligned}
\me|Y(t)-Y(t_0)|^2+\me\int_t^T|Z(t,s)-Z(t_0,s)|^2ds
\leq C|t-t_0|,
\end{aligned}
\end{equation*}
which is \eqref{estiyz5}.
\end{proof}

\section{The Euler method for SVIEs}
The aim of this section is to review the Euler method for SVIE \eqref{svie} under 
assumptions (A3)--(A4). For numerical
solutions to general 
SVIEs with singular kernels, one can refer to  \cite{ZhangXC08}.

For simplicity, throughout this paper, we assume that $\ds \D_i=|\pi|=\frac{T}{N}\leq 1$, for each $i=0,1,\cdots,N-1$. Our numerical scheme
still works for general uniform partition of $[0,T]$ (i.e., there exists a constant $K$, such that $K|\pi|\leq \D_j$, for any $j=0,1,\cdots,N-1$). We also need the following two fucntions $\t(\cdot)$ and 
$\pi(\cdot)$ defined on $[0,T)$ by
\begin{equation}\label{taupi}
\begin{aligned}
\t(t)=t_i,\,\,\,\pi(t)=i,\,\,\qq t\in [t_i,t_{i+1}),\,i=0,1,\cdots,N-1.
\end{aligned}
\end{equation}

\ms

The Euler method for SVIE \eqref{svie} is as follows: 
\begin{equation}\label{svieeuler}
\left\{\begin{aligned}
x^\pi(0)=&x^\pi(t_0)=\varphi(0),\\
x^\pi(t_{i+1})=&\varphi(t_{i+1})+\sum_{k=0}^i \Big(b(t_{i+1},t_k,x^\pi(t_k))\D_k+\sigma(t_{i+1},t_k,x^\pi(t_k))\D_kW\Big),\\
&\qq\qq\qq\qq\qq\qq\qq i=0,1,\cdots, N-1.\\
\end{aligned}
\right.
\end{equation}
In order to obtain the convergent speed, we introduce the following SVIE  related to \eqref{svieeuler}:
 \begin{equation}\label{svie20}
\begin{aligned}
x^\pi(t)=&\varphi(t)+\int_0^t b(t,\t(s),x^\pi(\t(s)))ds+\int_0^t \sigma(t,\t(s),x^\pi(\t(s)))dW(s),\,\,t\in[0,T].\\
\end{aligned}
\end{equation}

\ms

Now, we are in the step to obtain the convergent speed for the Euler method \eqref{svieeuler}. By SVIEs \eqref{svie} and \eqref{svie20}, one has, for any $t\in [0,T]$,
 \begin{equation*}\label{svieconver1}
\begin{aligned}
x(t)-x^\pi(t)=&\int_0^t \big(b(t,s,x(s))-b(t,\t(s),x^\pi(t(s)))\big)ds\\
  & +\int_0^t \big(\sigma(t,s,x(s))-\sigma(t,\t(s),x^\pi(\t(s)))\big) dW(s).\\
\end{aligned}
\end{equation*}
A direct calculation leads to
 \begin{equation}\label{svieconver2}
\begin{aligned}
&\me|x(t)-x^\pi(t)|^2\\
\leq&2T\me\int_0^t |b(t,s,x(s))-b(t,\t(s),x^\pi(\t(s)))|^2ds\\
   &+2\me\int_0^t |\sigma(t,s,x(s))-\sigma(t,\t(s),x^\pi(\t(s)))|^2 ds\\
   \leq&2(T+1)L^2\me\int_0^t\Big((s-\t(s))+|x(s)-x^\pi(s)|^2+|x^\pi(s)-x^\pi(\t(s))|^2\Big)ds.
\end{aligned}
\end{equation}
By the definition of $\t$ in \eqref{taupi} (suppose that $t\in [t_i, t_{i+1})$, $i=0,1,\cdots,N-1$),
 \begin{equation}\label{svieconver3}
\begin{aligned}
&\me\int_0^t (s-\t(s))ds=\sum_{k=0}^{i-1}\int_{t_k}^{t_{k+1}}(s-t_k)ds+\int_{t_i}^t(s-t_i)ds
=\sum_{k=0}^{i-1}\frac{\D_k^2}{2}+\frac{(t-t_i)^2}{2}\leq T|\pi|.
\end{aligned}
\end{equation}
Now, supposing that  $t\in [t_i, t_{i+1})$, we estimate $\me|x^\pi(t)-x^\pi(\t(t))|^2$. By SVIE \eqref{svie20}, one has
 \begin{equation*}\label{svieconver4}
\begin{aligned}
&x^\pi(t)-x^\pi(\t(t))=x^\pi(t)-x^\pi(t_i)\\
=&(\varphi(t)-\varphi(t_i))+\int_0^{t_i}\Big(b(t,\t(s),x^\pi(\t(s)))-b(t_i,\t(s),x^\pi(\t(s)))\Big)ds\\
  &+\int_0^{t_i}\Big(\sigma(t,\t(s),x^\pi(\t(s)))-\sigma(t_i,\t(s),x^\pi(\t(s)))\Big)dW(s)\\
  &+\int_{t_i}^t b(t,t_i,x^\pi(t_i))ds+\int_{t_i}^t \sigma(t,t_i,x^\pi(t_i))dW(s).\\
\end{aligned}
\end{equation*}
Then, under assumptions (A3)--(A4), it is easy to check that
\begin{equation}\label{svieconver5}
\begin{aligned}
\me|x^\pi(t)-x^\pi(\t(t))|^2
\leq&5\me|\varphi(t)-\varphi(t_i)|^2+5\Big(\int_0^{t_i}L\sqrt{t-t_i}ds\Big)^2
  +5\int_0^{t_i}L^2(t-t_i)ds\\
  &+5T\me\int_{t_i}^{t_{i+1}}|b(t,t_i,x^\pi(t_i))|^2ds+5\me\int_{t_i}^{t_{i+1}}|\sigma(t,t_i,x^\pi(t_i))|^2ds\\
 \leq & C|\pi|+ C\me|x^\pi(t_i)|^2|\pi|.
\end{aligned}
\end{equation}
For $\me|x^\pi(t_i)|^2$,  also by SVIE \eqref{svie20},
\begin{equation*}\label{svieconver6}
\begin{aligned}
\me|x^\pi(t)|^2
\leq &3\me|\varphi(t)|^2+6T\me\int_0^{t}|b(t,\pi(s),0)|^2ds+6TL^2\int_0^{t}|x^\pi(\pi(s))|^2ds\\
  &+6\me\int_0^{t}|\sigma(t,\pi(s),0)|^2ds+6L^2\int_0^{t}|x^\pi(\pi(s))|^2ds\\
\leq &C+6L^2(T+1)\int_0^t\me|x^\pi(\pi(s))|^2ds.\\
\end{aligned}
\end{equation*}
Setting $\ds g(t)=\sup_{s\in[0,t]}\me|x^\pi(s)|^2$, by Gronwall's inequality, one obtains
\begin{equation}\label{svieconver7}
\begin{aligned}
g(t)\leq Ce^{6L^2T(T+1)}, \,\mbox{for all}\,\, t\in [0,T].
\end{aligned}
\end{equation}
\eqref{svieconver5}, together with \eqref{svieconver7}, yields that
\begin{equation}\label{svieconver8}
\begin{aligned}
\me|x^\pi(t)-x^\pi(\t(t))|^2\leq C|\pi|,\, \forall t\in [0,T].
\end{aligned}
\end{equation}

By \eqref{svieconver2}, \eqref{svieconver3} and \eqref{svieconver8}, we have
\begin{equation*}\label{svieconver9}
\begin{aligned}
\me|x(t)-x^\pi(t)|^2\leq C|\pi|+2L^2(T+1)\int_0^t \me|x(s)-x^\pi(s)|^2ds,
\end{aligned}
\end{equation*}
which, by Gronwall's inequality, deduces that
\begin{equation*}\label{svieconver10}
\begin{aligned}
\sup_{t\in[0,T]}\me|x(t)-x^\pi(t)|^2\leq e^{2L^2(T+1)T}C|\pi|.
\end{aligned}
\end{equation*}

By the above analysis, we get the following convergence speed of the Euler method \eqref{svieeuler} for SVIE \eqref{svie}.
\begin{theorem}\label{svie-error}
Let {\rm(A3)--(A4)} hold. Then for $x(\cdot)$ and $x^\pi(\cdot)$ defined as in \eqref{svie} and \eqref{svieeuler}, respectively, there exists a constant $C$, depending only on $L$ and $T$, such that
\begin{equation}\label{svie-error1}
\begin{aligned}
\max_{0\leq i\leq N}\me|x(t_i)-x^\pi(t_i)|^2\leq C|\pi|.
\end{aligned}
\end{equation}
\end{theorem}

\ms

\section{The Euler method for BSVIEs}
In this section, we mainly present the Euler method to calculate the numerical solution to BSVIE \eqref{bsvie1}, and prove  the convergence speed of that method for  \eqref{bsvie1}. For  $1\leq k\leq N-1$, we present the Euler method for BSVIE \eqref{bsvie1} as follows:
\begin{equation}\label{ass1}
\left\{
\begin{aligned}
&Y^{k,\pi}(t_N)=g(t_k,x^\pi(T)),\\
&Y^{k,\pi}(t_l)=\me\Big(Y^{k,\pi}(t_{l+1})+f(t_k,t_l,x^\pi(t_l),Y^{l,\pi}(t_{l+1}),Z^{k,\pi}(t_{l}))\D_l\big|\mf_{t_l}\Big),\\
&Z^{k,\pi}(t_l)=\me\Big(\frac{\D_lW}{\D_l}\big(Y^{k,\pi}(t_{l+1})+f(t_k,t_l,x^\pi(t_l),Y^{l,\pi}(t_{l+1}),Z^{k,\pi}(t_{l}))\D_l\big)\big|\mf_{t_l}\Big),\\
 &\qq\qq\qq\qq\qq\qq\qq \qq\qq k\leq l\leq N-1.\\
\end{aligned}
\right.
\end{equation}
Here $\ds x^\pi(\cdot)\,(k\leq l\leq N-1)$ is defined by \eqref{svieeuler}.

In order to obtain the convergence speed, we introduce  $(Y^k(\cdot),Z^k(\cdot))\,(k=0,1,\cdots, N-1)$ solving the following BSDE: 
\begin{equation}\label{ass3}
\left\{\begin{aligned} 
& dY^k(s)=-f(t_{k},s,x(s),Y^l(s),Z^k(s))ds+Z^k(s)dW(s),\, s\in (t_l,t_{l+1}],\,k+1\leq l\leq N-1,\\
& dY^k(s)=-f(t_{k},s,x(s),Y^k(s),Z^k(s))ds+Z^k(s)dW(s),\, s\in [t_k,t_{k+1}],\\
& Y^k(T)=g(t_k,x(T)),\,Y^k(t_l)=Y^k(t_l+0),\,k+1\leq l\leq N-1,
\end{aligned}\right. 
\end{equation}
and  $\ds (Y^{k,\pi}(\cdot),\widehat Z^{k,\pi}(\cdot))\,(k=0,1,\cdots, N-1)$ solving the following BSDE: 
\begin{equation}\label{ass4}
\left\{\ba{ll} 
\ds Y^{k,\pi}(t_{l+1})-Y^{k,\pi}(t)=-f(t_{k},t_l,x^\pi(t_l),Y^{l,\pi}(t_{l+1}),\widehat Z_0^{k,\pi}(t_{l}))\D_l+\int_t^{t_{l+1}}\widehat Z^{k,\pi}(s)dW(s), \\
\ds                                  \qq\qq\qq\qq\qq \q  \qq\qq\qq\qq\, t\in (t_l,t_{l+1}],\,k+1\leq l\leq N-1,\\
\ds Y^{k,\pi}(t_{k+1})-Y^{k,\pi}(t)=-f(t_{k},t_k,x^\pi(t_k),Y^{k,\pi}(t_{k+1}),\widehat Z_0^{k,\pi}(t_{k}))\D_k+\int_t^{t_{k+1}}\widehat Z^{k,\pi}(s)dW(s),\\
\ds                                  \qq\qq\qq\qq\qq \q  \qq\qq\qq\qq \, t\in [t_k,t_{k+1}],\\
\ds Y^{k,\pi}(T)=g(t_k,x^\pi(T)),\,Y^{k,\pi}(t_l)=Y^{k,\pi}(t_l+0),\,k+1\leq l\leq N-1,\\
\ds \widehat Z_0^{k,\pi}(t_N)=0,\,\widehat Z_0^{k,\pi}(t_l)=\frac{1}{\D_l}\me\Big(\int_{t_l}^{t_{l+1}}\widehat Z^{k,\pi}(u)du|\mf_{t_l}\Big),\, k\leq l\leq N-1.
\ea\right. 
\end{equation}

\begin{remark}

(i) When $|\pi|<L^2$, BSDE \eqref{ass4} admits a unique solution.

(ii)
By \eqref{ass1} and BSDE \eqref{ass4}, in the  cases: (I) $f=f(t,s,x,y)$; (II) $f=f(t,s,x,z)$, 
we can easily check that, for any $k=0,1,\cdots,N-1$, $k+1\leq j\leq N-1$,
\begin{equation}\label{ass5}
\ba{c}
\ds Z^{k,\pi}(t_j)=\frac{1}{\D_j}\me\Big(\int_{t_j}^{t_{j+1}}\widehat Z^{k,\pi}(u)du|\mf_{t_j}\Big)=\widehat Z_0^{k,\pi}(t_j).
\ea
\end{equation}
\end{remark}

By the definition of $\t(\cdot)$ and $\pi(\cdot)$ in \eqref{taupi}, we can define 
$\ds Y^{\pi(t)}(t)=Y^k(t),\,Z^{\pi(t)}(s)=Z^k(s)$, 
and $\ds Y^{\pi(t),\pi}(\t(t))=Y^{k,\pi}(t_k),\,Z^{\pi(t),\pi}(\t(s))=Z^{k,\pi}(t_j)$,
for $\ds t\in [t_k,t_{k+1})$, 
$\ds k=0,1,\cdots,N-1$, and $s\geq t$, $s\in [t_j,t_{j+1}]$, $k\leq j\leq N-1$. 
The following result comes from \cite{Yong16}.
\begin{theorem}\label{yong1}
Let {\rm (A1)--(A4)} hold. Then, BSDE \eqref{ass3} admits a unique solution $(Y^{\pi(\cdot)}(\cdot), Z^{\pi(\cdot)}(\cdot))$, and 
\begin{equation*}
\begin{aligned}
\me\int_0^T|Y(t)-Y^{\pi(t)}(t)|^2dt+\me\int_0^T\int_t^T|Z(t,s)-Z^{\pi(t)}(s)|^2dsdt\leq K|\pi|,
\end{aligned}
\end{equation*}
where $K$ is a constant only depending on $L$ and $T$.
\end{theorem}

Now we state our main result on convergence speed of the Euler method \eqref{ass1} for BSVIE \eqref{bsvie1}.

\begin{theorem}\label{convergence}
Suppose that $f=f(t,s,x,y)$ or $f=f(t,s,x,z)$ in BSDE \eqref{ass3}, and let {\rm (A1)--(A4)} hold. Then
\begin{equation}\label{conver}
\begin{aligned}
\sup_{0\leq t\leq T}\me|Y(\t(t))-Y^{\pi(t),\pi}(\t(t))|^2+\me\int_0^T\int_t^T|Z(t,s)-Z^{\pi(t),\pi}(\t(s))|^2ds
\leq K|\pi|,
\end{aligned}
\end{equation}
where $K$ is a constant depending only on $L$ and $T$.

\end{theorem}

The proof of Theorem \ref{convergence} is lengthy, we split it into several lemmas.

\subsection{Regularity of $(Y^{\pi(\cdot)}(\cdot),Z^{\pi(\cdot)}(\cdot))$}

In this part, we mainly study the regularity of $Y^{\pi(\cdot)}(\cdot)$ and $Z^{\pi(\cdot)}(\cdot)$,
which is crucial in proving Theorem \ref{convergence}. First, we need the following lemma.
\begin{lemma}\label{gron}
(1) Suppose that $a_i\geq 0,\,b_i\geq 0,\,c_0> 0$ $(i=0,1,\cdots,N-1)$, and 
$$a_i\leq b_i+c_0\sum_{k=i+1}^{N-1}a_k.$$
Then 
\begin{equation}\label{gronwall1}
\begin{aligned}
a_i\leq b_i+c_0\sum_{k=i+1}^{N-1}(1+c_0)^{k-i-1}b_k.
\end{aligned}
\end{equation}

(2) Suppose that $b,\,K$ are positive constants, $\g=1 \mbox{ or } 2$, and for any $k=0,1,\cdots,N-1,\, k+1\leq j\leq N-1$
\begin{equation}\label{kk03}
\ba{c}
\ds a_{k,j}\leq b a_{k,j+1}+b|\pi| a_{j,j}+bK|\pi|^\g,\\
\ds    a_{k,k}\leq b a_{k,k+1}+bK|\pi|^\g.\\
\ea
\end{equation}
Then, the following holds true:
\begin{equation}\label{kk04}
\ba{c}
\ds a_{k,k}\leq b^{N-k-1}a_{k,N-1}+b^{N-k}|\pi|\sum_{l=0}^{N-k-3}(1+b|\pi|)^l a_{k+1+l,N-1}
              +bK|\pi|^\g \sum_{l=0}^{N-k-2}b^l (1+b|\pi|)^{l},\\
\ds a_{k,j}\leq  b^{N-1-j}a_{k,N-1}+b^{N-j}|\pi|\sum_{l=0}^{N-j-2}(1+b|\pi|)^l a_{j+l,N-1}
              +K|\pi|^\g \sum_{l=1}^{N-j-1}b^l (1+b|\pi|)^{l}.
\ea
\end{equation}


\end{lemma}

\begin{proof}
We prove \eqref{gronwall1}
by induction,
\begin{equation*}\label{abc1}
\begin{aligned}
a_{N-1}\leq& b_{N-1};\\
a_{N-2}\leq& b_{N-2}+c_0b_{N-1};\\
a_{N-3}\leq& b_{N-3}+c_0b_{N-2}+c_0(c_0+1)b_{N-1};\\
a_{N-4}\leq& b_{N-4}+c_0b_{N-3}+c_0(c_0+1)b_{N-2}+c_0(c_0+1)^2b_{N-1};\\
&\cdots\\
a_i\leq &b_i+c_0\sum_{k=i+1}^{N-1}(c_0+1)^{k-i-1}b_k.
\end{aligned}
\end{equation*}
Hence we obtain \eqref{gronwall1}. \eqref{kk04} can also proved by induction.
\end{proof}

\ms

The following Lemma is about the regularity of $Y^{\pi(\cdot)}(\cdot)$.

\begin{lemma}\label{regyk}
Suppose that {\rm(A1)--(A4)} hold true.  Then, for any $k=0,1,\cdots,N-1$, $k\leq j\leq N-1$ 
and $t\in [t_j,t_{j+1}]$, there exists a constant $C$, depending only on $L$ and $T$, such that
\begin{equation}\label{april05}
\begin{aligned}
\me\big(|Y^k(t)-Y^k(t_j)|^2+|Y^k(t)-Y^k(t_{j+1})|^2\big)\leq C|\pi|.
\end{aligned}
\end{equation}
\end{lemma}

\begin{proof}
For any $t\in [t_j,t_{j+1}]$,  by ESDE \eqref{ass3}, it is easy to see that
\begin{equation}\label{bai1}
\begin{aligned}
   &\me|Y^k(t)-Y^k(t_j)|^2\leq 2\me\int_{t_j}^t |f(t_k,s,x(s),Y^j(s),Z^k(s))|^2ds(t-t_j)+2\me\int_{t_j}^t |Z^k(s)|^2ds\\
\leq&8L^2\me\int_{t_j}^t\big(|f(t_k,s,0,0,0)|^2+|x(s)|^2+|Y^j(s)|^2+|Z^k(s)|^2\big)ds(t-t_j)+2\me\int_{t_j}^t|Z^k(s)|^2ds\\
\leq&K(t-t_j)+8L^2\me\int_{t_j}^t|Y^j(s)|^2ds(t-t_j)+\big(8L^2(t-t_j)+2\big)\me\int_{t_j}^t |Z^k(s)|^2ds.
\end{aligned}
\end{equation}

We now estimate $\me|Y^k(t)|^2$, for $k=0,1,\cdots,N-1$ and $t\in [t_k,t_{k+1}]$, 
which appears on the right side of \eqref{bai1}.
By It\^o's formula, 
\begin{equation*}\label{bai3}
\begin{aligned}
&\me|Y^k(t)|^2+\me\int_t^{t_{k+1}}|Z^k(s)|^2ds\\
\leq&\me|Y^k(t_{k+1})|^2+\me\int_t^{t_{k+1}}\Big((2L+3L^2+1)|Y^k(s)|^2ds
            +|f(t_k,s,0,0,0)|^2+|x(s)|^2+\frac 1 2|Z^k(s)|^2\Big)ds\\
\leq&\me|Y^k(t_{k+1})|^2+(2L+3L^2+1)\me\int_t^{t_{k+1}}|Y^k(s)|^2ds
           +(L^2+K)(t_k-t)+\frac 1 2\me\int_t^{t_k}|Z^k(s)|^2ds.\\
\end{aligned}
\end{equation*}
Consequently, by Gronwall's inequality,
\begin{equation}\label{bai4}
\begin{aligned}
\me|Y^k(t)|^2\leq e^{(2L+3L^2+1)(t_{k+1}-t)}\Big(\me|Y^k(t_{k+1})|^2+(L^2+K)|\pi|\Big).\\
\end{aligned}
\end{equation}

Now, we estimate $\me|Y^k(t)|^2$, for any $t\in[t_j,t_{j+1}]$ $(k+1\leq j\leq N-1)$, which appears in \eqref{bai4}. With the similar 
calculus to that of  \eqref{bai4}, one can obtain
\begin{equation*}\label{bai5}
\begin{aligned}
&\me|Y^k(t)|^2+\me\int_t^{t_{j+1}}|Z^k(s)|^2ds\\
\leq&\me|Y^k(t_{j+1})|^2+2\me\int_t^{t_{j+1}}|Y^k(s)|\big(|f(t_k,s,0,0,0)|+L|x(s)|+L|Y^j(s)|+L|Z^k(s)|\big)ds\\
\leq&\me|Y^k(t_{j+1})|^2+(4L^2+1)\me\int_t^{t_{j+1}}|Y^k(s)|^2ds\\
       &+\me\int_t^{t_{j+1}}\big(|f(t_k,s,0,0,0)|^2+|x(s)|^2+|Y^j(s)|^2+\frac1 2|Z^k(s)|^2\big)ds\\
\leq&\me|Y^k(t_{j+1})|^2+(4L^2+1)\me\int_t^{t_{j+1}}|Y^k(s)|^2ds\\
     &+(L^2+K)(t_{j+1}-t)+\me\int_t^{t_{j+1}}|Y^j(s)|^2ds+\frac1 2\me\int_t^{t_{j+1}}|Z^k(s)|^2ds.
\end{aligned}
\end{equation*}
Also, by Gronwall's inequality, one has,
\begin{equation}\label{bai6}
\begin{aligned}
\me|Y^k(t)|^2\leq e^{(4L^2+1)(t_{j+1}-t)}\Big(\me|Y^k(t_{j+1})|^2+(L^2+K)|\pi|
      +\me\int_t^{t_{j+1}}|Y^j(s)|^2ds\Big).
\end{aligned}
\end{equation}
Set $\bar L =\max\{2L+3L^2+1,\,4L^2+1\},\,\, \a=e^{\bar L|\pi|},\,\, 
\bar K=L^2+K$, and $J_{k,j}=\ds \sup_{t_j\leq t<t_{j+1}}\me|Y^k(t)|^2,\,\ j\geq k+1.$
Then,
by \eqref{bai4}, \eqref{bai6} and Lemma  \ref{gron}, it comes out
\begin{equation}\label{bai9}
\begin{aligned}
J_{k,j}\leq& \a^{N-1-j}J_{k,N-1}+\a^{N-j}|\pi|\sum_{l=0}^{N-2-j}(1+\a |\pi|)^l J_{j+l,N-1}
+\bar K|\pi|\sum_{l=1}^{N-1-j}\a^l(1+\a|\pi|)^l.\\
\end{aligned}
\end{equation}
Since for all $k\leq N-2$,
\begin{equation}\label{bai10}
\begin{aligned}
   &J_{k,N-1}=\sup_{t_{N-1}\leq t\leq T}\me |Y^k(t)|^2\\
\leq& e^{\bar L|\pi|}\bigg(\me|g(t_k,x(T))|^2+\bar K|\pi|+\int_{t_{N-1}}^T|Y^{N-1}(s)|^2ds\bigg)\\
\leq &  e^{\bar L|\pi|}\Big(L^2T+L^2\me|x(T)|^2+\bar K|\pi|+|\pi|J_{N-1,N-1}\Big),\\
\end{aligned}
\end{equation}
and 
\begin{equation}\label{bai11}
\begin{aligned}
 &J_{N-1,N-1}=\sup_{t_{N-1}\leq t\leq T}\me |Y^{N-1}(t)|^2
 \leq e^{\bar L|\pi|}\Big(\me |g(t_{N-1},x(T))|^2+\bar K|\pi|\Big)\\
 \leq &e^{\bar L|\pi|}\Big(L^2T+L^2\me |x(T)|^2+\bar K|\pi|\Big)
 \leq C<\infty,
\end{aligned}
\end{equation}
\eqref{bai10}, together with \eqref{bai11}, yields that, for all $k=0,1,\cdots,N-1$,
\begin{equation}\label{bai12}
\begin{aligned}
 J_{k,N-1}
 \leq C<\infty.
\end{aligned}
\end{equation}
Now, we estimate the right side of \eqref{bai9} term by term.
\begin{equation}\label{bai13}
\begin{aligned}
\a^{N-i-j}\leq \a^N=e^{\bar L|\pi|N}=e^{\bar LT}<\infty,
\end{aligned}
\end{equation}
\begin{equation}\label{bai14}
\begin{aligned}
&\a^{N-j}|\pi|\sum_{l=0}^{N-2-j}(1+\a |\pi|)^l J_{j+l,N-1}
\leq  Ce^{\bar LT}\frac{(1+\a|\pi|)^{N}}{\a}\\
=    &C\big(1+\frac{Te^{\bar LT}}{N}\big)^N \leq  C e^{Te^{\bar L T}}<\infty,
\end{aligned}
\end{equation}
%
%
and
\begin{equation}\label{bai16}
\begin{aligned}
  &\bar K |\pi|\sum_{l=1}^{N-1-j}\a^l(1+\a|\pi|)^l
\leq  \bar K|\pi|\frac{\a^{N}(1+\a|\pi|)^{N}}{\a+\a^2|\pi|-1}\\
\leq & \bar K|\pi|\frac{\a^{N}(1+\a|\pi|)^{N}}{\bar L|\pi|}
\leq \frac {\bar K}{ \bar L} e^{\bar LT}e^{Te^{\bar LT}}<\infty.
\end{aligned}
\end{equation}
Hence, \eqref{bai9}, together with \eqref{bai13}--\eqref{bai16}, leads to, 
for any $k=0,1,\cdots,N-1$ and $j\geq k+1$,
\begin{equation}\label{bai17}
\begin{aligned}
J_{k,j}=\ds \sup_{t_j\leq t<t_{j+1}}\me|Y^k(t)|^2<\infty.
\end{aligned}
\end{equation}
Thereafter
$$J_{k,k}\leq \a J_{k,k+1}+\a \bar K|\pi|<\infty.$$
Furthermore, the second term of the right side in \eqref{bai1} turns into
\begin{equation}\label{bai18}
\begin{aligned}
\me\int_{t_j}^t|Y^j(s)|^2ds(t-t_j)\leq C|t-t_j|^2\leq C|t-t_j|.
\end{aligned}
\end{equation}

\ms

Now, we need to estimate $\me|Z^k(t)|^2$, for any $t\in [t_j, t_{j+1}]$, 
which appears on the third term of the right side in \eqref{bai1}.
Since $(D_{\th}Y^{\pi(\cdot)}(\cdot),D_{\th}Z^{\pi(\cdot)}(\cdot))$, the Malliavin derivative of 
$(Y^{\pi(\cdot)}(\cdot),Z^{\pi(\cdot)}(\cdot))$, satisfies the following BSDE (\cite[Proposition 5.3]{ElKaroui-Peng-Quenez97}): 
\begin{equation}\label{april1}
\left\{
\begin{aligned}
&D_{\th}Y^k(t_{j+1})-D_\th Y^k(t)\\
=&\int_t^{t_{j+1}}\Big(-f_x(t_k,s,x(s),Y^j(s),Z^k(s))D_\th x(s)
           -f_y(t_k,s,x(s),Y^j(s),Z^k(s))D_\th Y^j(s)\\
   &\q-f_z(t_k,s,x(s),Y^j(s),Z^k(s))D_\th Z^k(s)\Big)ds 
        +\int_t^{t_{j+1}}D_\th Z^k(s)dW(s),\q t\in[t_j,t_{j+1}],\,\theta\in [0,t],\\
&Z^k(t)=D_t Y^k(t),  t\in [t_k,T].
\end{aligned}
\right.
\end{equation}
By It\^o's formula, for any $t\in [t_k, t_{k+1}]$, one has
\begin{equation*}\label{april2}
\begin{aligned}
&\me|D_\th Y^k(t)|^2+\me\int_t^{t_{k+1}}|D_\th Z^k(s)|^2ds\\
\leq & \me\bigg\{|D_\th Y^k(t_{k+1})|^2 +(2L+3L^2)\int_t^{t_{k+1}}|D_\th Y^k(s)|^2ds\\
     &\q +\int_t^{t_{k+1}}|D_\th x(s)|^2ds+\frac 1 2\int_t^{t_{k+1}}|D_\th Z^k(s)|^2ds\bigg\},
\end{aligned}
\end{equation*}
by Gronwall's inequality, which deduces that, 
$$\me|D_\th Y^k(t)|^2\leq e^{(2L+3L^2)(t_{k+1}-t)}\Big(\me|D_\th Y^k(t_{k+1})|^2+K|\pi|\Big).$$
Similarly, for any $t\in [t_j, t_{j+1}]$,
\begin{equation}\label{april200}
\begin{aligned}
&\me|D_\th Y^k(t)|^2+\me\int_t^{t_{j+1}}|D_\th Z^k(s)|^2ds\\
\leq & \me\bigg\{|D_\th Y^k(t_{j+1})|^2 +4L^2\int_t^{t_{j+1}}|D_\th Y^k(s)|^2ds\\
      &\q+\int_t^{t_{j+1}}\big(|D_\th x(s)|^2+|D_\th Y^j(s)|^2\big)ds+\frac 1 2\int_t^{t_{j+1}}|D_\th Z^k(s)|^2ds\bigg\}.
\end{aligned}
\end{equation}
Therefore, similar to \eqref{bai17}, by virtue of Lemma \ref{gron}, we has
\begin{equation}\label{april3}
\begin{aligned}
\sup_{t_j\leq t\leq t_{j+1}}\me |D_\th Y^k(t)|^2<\infty.
\end{aligned}
\end{equation}
By setting $\th=t$ in \eqref{april3}, one gets, for any $t\in [t_k, T]$,
\begin{equation}\label{april4}
\begin{aligned}
\me|Z^k(t)|^2=\me |D_t Y^k(t)|^2<\infty.
\end{aligned}
\end{equation}
\eqref{bai1}, together with \eqref{bai18} and \eqref{april4}, yields that
\begin{equation}\label{april5}
\begin{aligned}
\me|Y^k(t)-Y^k(t_j)|^2\leq C|t-t_j|\leq C|\pi|.
\end{aligned}
\end{equation}
Similarly, we can get 
$$\me|Y^k(t)-Y^k(t_{j+1})|^2\leq C|t-t_{j+1}|\leq C|\pi|.$$
That completes the proof.
\end{proof}


With this result at hand we can conclude:

\begin{proposition}\label{yyk}
Suppose that {\rm(A1)--(A4)} hold true. Then, there exists a constant $K$,  such that, for any $k=0,1,\cdots,N-1$, 
$$\me|Y(t_k)-Y^k(t_k)|^2+\me\int_{t_k}^T|Z(t_k,s)-Z^k(s)|^2ds\leq K|\pi|.$$

\end{proposition}

\begin{proof}
Setting $h(t,s,z)=f(t,s,x(s),Y(s),z)$ and $\bar h(t,s,z)=f(\t(t),s,x(s),Y^{\pi(s)}(s),z)$, 
by \cite[Corrolary 3.6]{Yong08}, we have 
\begin{equation}\label{esti5}
\begin{aligned}
&\me|Y(t_k)-Y^k(t_k)|^2+\me\int_{t_k}^T|Z(t_k,s)-Z^k(s)|^2ds\\
\leq & K\me\Big(\int_{t_k}^T\big|f(t_k,s,x(s),Y(s),Z(t_k,s))-f(\t(t_k),s,x(s),Y^{\pi(s)}(s),Z(t_k,s))\big|ds\Big)^2\\
\leq& K\sum_{j=k}^{N-1}\me\int_{t_{j}}^{t_{j+1}}\big(|Y(s)-Y(t_j)|^2+|Y^j(s)-Y^j(t_j)|^2+|Y(t_j)-Y^j(t_j)|^2\big)ds\\
\leq& K|\pi|+K|\pi|\sum_{j=k+1}^{N-1}\me|Y(t_j)-Y^j(t_j)|^2+K|\pi|\me|Y(t_k)-Y^k(t_k)|^2.\\
\end{aligned}
\end{equation}
Here we apply Lemma \ref{estimateyz} and Lemma \ref{regyk}.
Taking $\ds N> 2KT\, (N\in\dbN)$, then $\ds K|\pi|\leq \frac{1}{2}$, and
\begin{equation}\label{esti6}
\begin{aligned}
&\frac{1}{2}\me|Y(t_k)-Y^k(t_k)|^2
\leq K|\pi|+ K|\pi|\sum_{j=k+1}^{N-1}\me|Y(t_j)-Y^j(t_j)|^2.\\
\end{aligned}
\end{equation}

Denote $\ds a_k=\me|Y(t_k)-Y^k(t_k)|^2$, $b_i=2K|\pi|$ and $c_0=2K|\pi|$.
Therefore, by Lemma \ref{gron}, for any $k=0,1,\cdots, N-1$,
\begin{equation}\label{conver2}
\begin{aligned}
&\me|Y(t_k)-Y^k(t_k)|^2=
a_k\leq b_k+c_0\sum_{i=k+1}^{N-1}(1+c_0)^{i-k-1}b_i\\
\leq &2K|\pi|+2K|\pi|(c_0+1)^N
\leq 2K|\pi|+2K|\pi| e^{2KT}
\leq K|\pi|.
\end{aligned}
\end{equation}
Combining \eqref{esti5} with \eqref{conver2}, we can have $\ds \me\int_{t_k}^T|Z(t_k,s)-Z^k(s)|^2ds\leq K|\pi|$.
That completes the proof.
\end{proof}

\ms

In the following part, we mainly provide the regularity of $Z^{\pi(\cdot)}(\cdot)$. Such a regularity, 
combining with that for $x(\cdot)$ and $Y^{\pi(\cdot)}(\cdot)$, can derive the rate of convergence of the 
Euler method \eqref{ass1}.
 We present that regularity in two different 
cases: (I) $f=f(t,s,x,y)$; (II) $f=f(t,s,x,z)$. Here, we borrow some idea
from \cite{Hu-Nualart-Song11}.

\begin{lemma}\label{regzka}
Suppose that $f=f(t,s,x,y)$ in BSDE \eqref{ass3}, and  {\rm(A1)--(A4)} hold true. 
Then, for any $k=0,1,\cdots,N-1$, $k\leq j\leq N-1$ 
and $s\in [t_j,t_{j+1}]$, there exists a constant $C$, such that
\begin{equation}\label{april06a}
\begin{aligned}
\me|Z^k(s)-Z^k(t_j)|^2\leq K|\pi|.
\end{aligned}
\end{equation}
\end{lemma}

\begin{proof}
We divide the proof into two steps.

{\bf Step 1.}
For any $k=0,1,\cdots,N-1$,
$k\leq j\leq N-1$ and $s\in [t_j,t_{j+1}]$,
by \eqref{april1}, one gets
\begin{equation}\label{april7a}
\begin{aligned}
&Z^k(s)-Z^k(t_j)=D_sY^k(s)-D_{t_j}Y^k(t_j)\\
=&\big(D_sY^k(s)-D_{t_j}Y^k(s)\big)+\big(D_{t_j}Y^k(s)-D_{t_j}Y^k(t_j)\big).
\end{aligned}
\end{equation}

 We claim that
 \begin{equation}\label{cl1}
 \begin{aligned}
 \me|D_sY^k(s)-D_{t_j}Y^k(s)|^2\leq K|\pi|.
 \end{aligned}
 \end{equation}
 Indeed,
 by \eqref{april1}, for $\th_1, \th_2\in [t_j,t_{j+1}]$, $\th_2\leq \th_1\leq s$,
\begin{equation*}\label{april8a}
\begin{aligned}
&\me|D_{\th_1}Y^k(s)-D_{\th_2}Y^k(s)|^2+\me\int_s^{t_{j+1}}|D_{\th_1}Z^k(t)-D_{\th_2}Z^k(t)|^2dt\\
=&\me|D_{\th_1}Y^k(t_{j+1})-D_{\th_2}Y^k(t_{j+1})|^2\\
   &+2\me\int_s^{t_{j+1}}\Big\lan  D_{\th_1}Y^k(t)-D_{\th_1}Y^k(t), f_x(D_{\th_1}x(t)-D_{\th_2}x(t))
       +f_y(D_{\th_1}Y^j(t)-D_{\th_2}Y^j(t))  \Big\ran dt\\
\leq &\me|D_{\th_1}Y^k(t_{j+1})-D_{\th_2}Y^k(t_{j+1})|^2
            +2L^2\me\int_s^{t_{j+1}}|D_{\th_1}Y^k(t)-D_{\th_2}Y^k(t)|^2dt\\
       &  +  \me\int_s^{t_{j+1}}|D_{\th_1}x(t)-D_{\th_2}x(t)|^2dt
           +\me\int_s^{t_{j+1}}|D_{\th_1}Y^j(t)-D_{\th_2}Y^j(t)|^2dt.\\
\end{aligned}
\end{equation*}
Hence, by Lemma \ref{sviex} and Gronwall's inequality,
 one has, for any $s\in [t_j,t_{j+1}]$,
\begin{equation*}\label{april12a}
\begin{aligned}
&\me|D_{\th_1}Y^k(s)-D_{\th_2}Y^k(s)|^2\\
\end{aligned}
\end{equation*}
\begin{equation*}
\begin{aligned}
\leq &e^{2L^2(t_{j+1}-s)}\Big(\me|D_{\th_1}Y^k(t_{j+1})-D_{\th_2}Y^k(t_{j+1})|^2
              + K|\th_1-\th_2|(t_{j+1}-s)\\
       &  +\me\int_s^{t_{j+1}}|D_{\th_1}Y^j(t)-D_{\th_2}Y^j(t)|^2dt\Big)\\
\leq &e^{2L^2(t_{j+1}-s)}\Big(\me|D_{\th_1}Y^k(t_{j+1})-D_{\th_2}Y^k(t_{j+1})|^2
              + K|\pi|^2\\
       &  +\me\int_s^{t_{j+1}}|D_{\th_1}Y^j(t)-D_{\th_2}Y^j(t)|^2dt\Big).\\
\end{aligned}
\end{equation*}
Similarly, for any $s\in [t_k,t_{k+1}]$,
\begin{equation*}\label{april13a}
\begin{aligned}
&\me|D_{\th_1}Y^k(s)-D_{\th_2}Y^k(s)|^2
\leq e^{2L^2(t_{k+1}-s)}\Big(\me|D_{\th_1}Y^k(t_{k+1})-D_{\th_2}Y^k(t_{k+1})|^2
              + K|\pi|^2\Big).\\
\end{aligned}
\end{equation*}
By Lemma \ref{gron}, with the similar procedure used in the proof of Lemma \ref{regyk},  
one can get,  for any $s\in [t_j,t_{j+1}]$,
\begin{equation}\label{april130a}
\begin{aligned}
&\sup_{t_j\leq s\leq t_{j+1}}\me|D_{\th_1}Y^k(s)-D_{\th_2}Y^k(s)|^2
\leq  K|\pi|.\\
\end{aligned}
\end{equation}
Setting $\th_1=s,\, \th_2=t_j$, one easily obtains \eqref{cl1}.

\ms

{\bf Step 2.} We claim that, for any $s\in [t_j, t_{j+1}]$, 
\begin{equation}\label{april14a}
\begin{aligned}
\me|D_{t_j}Y^k(s)-D_{t_j}Y^k(t_j)|\leq K|s-t_j|.
\end{aligned}
\end{equation}

For any $\th\leq t_j$, $t\in [t_j,t_{j+1}]$, by virtue of Eq. \eqref{ass3},
\begin{equation*}\label{april15a}
\begin{aligned}
D_\th Y^k(t)=&\me\Big(D_\th Y^k(T)+\int_t^T F(s)ds\Big|\mf_t\Big),\\
\end{aligned}
\end{equation*}
where
\begin{equation*}\label{april16a}
\begin{aligned}
\int_t^T F(s)ds=&\int_t^{t_{j+1}}f_x(t_k,s,x(s),Y^j(s))D_\th x(s)
          +f_y(t_k,s,x(s),Y^j(s))D_\th Y^j(s)ds\\
           &+\sum_{l=j+1}^{N-1}\int_{t_l}^{t_{l+1}}f_x(t_k,s,x(s),Y^l(s))D_\th x(s)
          +f_y(t_k,s,x(s),Y^l(s))D_\th Y^l(s)ds.\\
\end{aligned}
\end{equation*}
Then
\begin{equation}\label{april17a}
\begin{aligned}
&D_\th Y^k(t)-D_\th Y^k(t_j)\\
=&\bigg\{\me\big(D_\th Y^k(T)|\mf_t\big)-\me\big(D_\th Y^k(T)|\mf_{t_j}\big)\bigg\}
           +\bigg\{\me\Big(\int_t^T F(s)ds|\mf_t\Big)-\me\Big(\int_t^T F(s)ds|\mf_{t_j}\Big)\bigg\}\\
:=&I_1+I_2.\\
\end{aligned}
\end{equation}

For $I_1$, since
\begin{equation*}\label{april19a}
\begin{aligned}
D_\th Y^k(T)=\me D_\th Y^k(T)+\int_0^T\me\big(D_sD_\th Y^k(T)|\mf_s\big)dW(s),
\end{aligned}
\end{equation*}
by Lemma \ref{sviex}, one can have
\begin{equation}\label{april20a}
\begin{aligned}
&\me I_1^2=\me \int_{t_j}^t\Big|\me\big(D_sD_\th Y^k(T)|\mf_s\big)\Big|^2ds
   \leq  \me \int_{t_j}^t\big|D_sD_\th Y^k(T)\big|^2ds\\
    =   &    \me \int_{t_j}^t\big|D_sD_\th g(t_k,x(T))\big|^2ds
    =   \me \int_{t_j}^t\big|g_{xx}D_\th x(T)D_u x(T)+g_x D_sD_\th x(T)\big|^2ds
    \leq K|t-t_j|.
\end{aligned}
\end{equation}
%

For $I_2$,
\begin{equation}\label{april18a}
\begin{aligned}
I_2=&\bigg\{\me\Big(\int_t^T F(s)ds|\mf_t\Big)-\me\Big(\int_{t_j}^T F(s)ds|\mf_t\Big)\bigg\}\\
       &\qq+\bigg\{\me\Big(\int_{t_j}^T F(s)ds|\mf_t\Big)-\me\Big(\int_{t_j}^T F(s)ds|\mf_{t_j}\Big)\bigg\}\\
     :=&I_{21}+I_{22}.  
\end{aligned}
\end{equation}
By Lemma \ref{sviex} and \eqref{april3}, it is easy to check that
\begin{equation}\label{april21a}
\begin{aligned}
&\me I_{21}^2=\me \Big|\int_t^{t_{j+1}}\big(f_x(t_k,s,x(s),Y^j(s))D_\th x(s)
          +f_y(t_k,s,x(s),Y^j(s))D_\th Y^j(s)\big)ds \Big|^2\\
   \leq & K(t-t_j)\me \int_{t_j}^t\big(\big|D_\th x(s)\big|^2+\big|D_\th Y^j(s)\big|^2\big)ds             
              \leq K|t-t_j|.
\end{aligned}
\end{equation}
For the $I_{22}$ part,
by Clark-Ocone representation formula,
$$
\int_{t_j}^T F(s)ds=\me\Big(\int_{t_j}^T F(s)ds\Big)+\int_0^T\me\Big(D_u\int_{t_j}^T F(s)ds\Big|\mf_u\Big)dW(u),
$$
it admits the following representation:
$$ I_{22}=\int_{t_j}^t\me\Big(D_u\int_{t_j}^T F(s)ds\Big|\mf_u\Big)dW(u).$$
It is easy to check that
\begin{equation*}\label{april22a}
\begin{aligned}
&D_u\int_{t_j}^T F(s)ds=\int_{t_j}^T D_uF(s)ds\\
=&\sum_{l=j}^{N-1}\int_{t_l}^{t_{l+1}}D_u\big(f_x(t_k,s,x(s),Y^l(s))D_\th x(s)
          +f_y(t_k,s,x(s),Y^l(s))D_\th Y^l(s)\big)ds\\
=&\sum_{l=j}^{N-1}\int_{t_l}^{t_{l+1}}\Big(f_{xx}D_\th x(s)D_u x(s)+f_{xy}D_\th x(s)D_u Y^l(s)
             +f_xD_uD_\th x(s)\\
   &\qq\qq+  f_{yx}D_\th Y^l(s)D_u x(s)+f_{yy}D_\th Y^l(s)D_u Y^l(s)
             +f_yD_uD_\th Y^l(s)        \Big)ds.\\     
\end{aligned}
\end{equation*}
Therefore,
\begin{equation}\label{april23a}
\begin{aligned}
&\me\Big|D_u\int_{t_j}^T F(s)ds\Big|^2\leq T\me\int_{t_j}^T |D_uF(s)|^2ds\\
\leq&K \sum_{l=j}^{N-1}\me\int_{t_l}^{t_{l+1}}\Big(|D_\th x(s)D_u x(s)|^2+|D_\th x(s)D_u Y^l(s)|^2
             +|D_uD_\th x(s)|^2\\
   &+  |D_\th Y^l(s)D_u x(s)|^2+|D_\th Y^l(s)D_u Y^l(s)|^2
             +|D_uD_\th Y^l(s)|^2        \Big)ds\\
\leq&K \sum_{l=j}^{N-1}\me\int_{t_l}^{t_{l+1}}\Big(|D_\th x(s)|^4+|D_u x(s)|^4+|D_u Y^l(s)|^4
            +  |D_\th Y^l(s)|^4\\
   &\qq\qq +|D_uD_\th x(s)|^2
             +|D_uD_\th Y^l(s)|^2        \Big)ds.\\                  
\end{aligned}
\end{equation}
Now, we estimate each term on the right side of the above inequality. By It\^o's formula, 
\begin{equation*}\label{april24a}
\begin{aligned}
&\me|D_\th Y^k(t)|^4 +6\me\int_t^{t_{j+1}}|D_\th Y^k(s)|^2 |D_\th Z^k(s)|^2 ds  \\
=&\me|D_\th Y^k(t_{j+1})|^4+4\me\int_t^{t_{j+1}}|D_\th Y^k(s)|^2 
      \big\langle  D_\th Y^k(s),   f_xD_\th x(s)+f_y D_\th Y^j(s)\big\rangle ds\\  
\leq &\me|D_\th Y^k(t_{j+1})|^4+4L\me\int_t^{t_{j+1}}\Big( \frac 3 2|D_\th Y^k(s)|^4+4|D_\th x(s)|^4
          +4 |D_\th Y^j(s)|^4\Big) ds\\ 
\leq &\me|D_\th Y^k(t_{j+1})|^4+6L\me\int_t^{t_{j+1}}|D_\th Y^k(s)|^4ds 
        +K|\pi|+K\me\int_t^{t_{j+1}}|D_\th Y^j(s)|^4 ds.\\                                  
\end{aligned}
\end{equation*}
Thus, by Lemma \ref{gron}, one get
\begin{equation}\label{april25a}
\begin{aligned}
\sup_{t_j\leq t\leq t_{j+1}}\me|D_\th Y^k(t)|^4<\infty.                              
\end{aligned}
\end{equation}

\ms

For any $u\leq t_j\leq t\leq t_{j+1}$,
\begin{equation}\label{april26a}
\begin{aligned}
&D_uD_\th Y^k(t_{j+1})-D_uD_\th Y^k(t)   \\
=&\int_t^{t_{j+1}}\Big(f_{xx}D_\th x(s)D_u x(s)+f_{xy}D_\th x(s)D_u Y^j(s)
             +f_xD_uD_\th x(s)\\
   &+  f_{yx}D_\th Y^j(s)D_u x(s)+f_{yy}D_\th Y^j(s)D_u Y^j(s)
             +f_yD_uD_\th Y^j(s)  \Big)ds\\
       &+\int_t^{t_{j+1}}D_uD_\th Z^k(s)dW(s).\\                      
\end{aligned}
\end{equation}
Hence, by It\^o's formula,
\begin{equation*}\label{april27a}
\begin{aligned}
&\me|D_uD_\th Y^k(t)|^2+\me\int_t^{t_{j+1}}|D_uD_\th Z^k(s)|^2ds   \\
\leq &\me|D_uD_\th Y^k(t_{j+1})|^2+\me\int_t^{t_{j+1}}\Big(6L^2|D_uD_\th Y^k(s)|^2
            +|D_\th x(s)D_u x(s)|^2+|D_ \th x(s)D_u Y^j(s)|^2\\
      &  + |D_uD_\th x(s)|^2
        +  |D_\th Y^j(s)D_u x(s)|^2+|D_\th Y^j(s)D_u Y^j(s)|^2+
             |D_uD_\th Y^j(s)|^2  \Big)ds\\
\leq &\me|D_uD_\th Y^k(t_{j+1})|^2+6L^2\me\int_t^{t_{j+1}}|D_uD_\th Y^k(s)|^2ds
           +K|\pi|+\me\int_t^{t_{j+1}}  |D_uD_\th Y^j(s)|^2ds\\               
\end{aligned}
\end{equation*}
Also, by Lemma \ref{gron}, we have
\begin{equation}\label{april28a}
\begin{aligned}
\sup_{t_j\leq t\leq t_{j+1}}\me|D_uD_\th Y^k(t)|^2<\infty.           
\end{aligned}
\end{equation}
Therefore,
 \eqref{april23a}, together with \eqref{april25a} and \eqref{april28a}, yields that
\begin{equation*}\label{april29a}
\begin{aligned}
\me\Big|D_u\int_{t_j}^T F(s)ds\Big|^2<\infty.           
\end{aligned}
\end{equation*}
Furthermore,
\begin{equation}\label{april30a}
\begin{aligned}
&\me|I_{22}|^2=\me\int_{t_j}^t\Big|\me\Big(D_u\int_{t_j}^T F(s)ds\Big|\mf_u\Big)\Big|^2du
\leq  \me\int_{t_j}^t\Big|D_u\int_{t_j}^T F(s)ds\Big|^2du\leq K|t-t_j|.
\end{aligned}
\end{equation}

Finally, by \eqref{april17a}--\eqref{april21a} and  \eqref{april30a},
one gets
$$\me|D_\th Y_k(t)-D_\th Y_k(t_j)|^2\leq K|t-t_j|,$$
which deduces \eqref{april14a} by setting $\th=t_j$. 
Now combining \eqref{april7a} with \eqref{cl1} and \eqref{april14a}, we have the regularity 
of $Z$ \eqref{april06a}.
\end{proof}

\ms

The following regularity of $Z^{\pi(\cdot)}(\cdot)$ is in the case: $f=f(t,s,x,z)$.

\begin{lemma}\label{regzkb}
Suppose that $f=f(t,s,x,z)$ in BSDE \eqref{ass3}, and  {\rm(A1)--(A4)} hold true. 
Then, for any $k=0,1,\cdots,N-1$, $k\leq j\leq N-1$ 
and $s\in [t_j,t_{j+1}]$, there exists a constant $C$, such that
\begin{equation}\label{april06b}
\begin{aligned}
\me|Z^k(s)-Z^k(t_j)|^2\leq C|\pi|.
\end{aligned}
\end{equation}
\end{lemma}

We need the following lemma to prove the above result. 
\begin{lemma}\label{wang65}
Let {\rm(A1)} hold, and for any $k=0,1,\cdots,N-1$, $\Psi_k(\cdot)$ and $\Phi_k(\cdot)$ solve the following SDEs
\begin{equation}\label{wang1}
\left\{
\begin{split}
 d\Psi_k(t)&=\Psi_k(t)f_z(t_k,t,x(t),Z^k(t)) dW(t),\q   t\in [0,T), \\
 \displaystyle
 \Psi(0)&=I_n
\end{split}
\right.
\end{equation}
and 
\begin{equation}\label{wang2}
\left\{
\begin{aligned}
 d\Phi_k(t)&=\big(f_z(t_k,t,x(t),Z^k(t))\big)^2\Phi_k(t) dt\\
      &\q-f_z(t_k,t,x(t),Z^k(t))\Phi_k(t) dW(t),\q   t\in [0,T), \\
 \displaystyle
 \Phi(0)&=I_n,
\end{aligned}
\right.
\end{equation}
respectively.  Then, for any $p\geq 2$, 
\begin{equation}\label{wang66}
\begin{aligned}
\me\big(\sup_{0\leq t\leq T}|\Psi_k(t)|^p\big)+\me\big(\sup_{0\leq t\leq T}|\Phi_k(t)|^p\big)\leq C,\\
\end{aligned}
\end{equation}
\begin{equation}\label{wang41}
\begin{aligned}
\me\big(\sup_{s\leq t\leq T}\big|\Phi_k(s)\Psi_k(t)\big|^p\big)
 + \me\big(\sup_{s\leq t\leq T}|\Phi_k(t)\Psi_k(s)|^p\big)\leq C,\\
\end{aligned}
\end{equation}
\begin{equation}\label{wang45}
\begin{aligned}
\me|(\Phi_k(t)-\Phi_k(s))\Psi_k(T_0)|^p\leq C|t-s|^{\frac{p}{2}},\,\,t,s\leq T_0\leq T;
\end{aligned}
\end{equation}
and for any $p\in [2,2p_0)$,
\begin{equation}\label{wang45a}
\begin{aligned}
\me\big(\sup_{\theta,s\leq t\leq T}|D_{\theta}(\Phi_k (s)\Psi_k (t))|^p\big)\leq C,
\end{aligned}
\end{equation}
where $C$ depends only on $p,\,L$ and $T$.
\end{lemma}

\begin{proof}
First of all,
for any $x_0\in \dbR^n$, set $x(\cdot)=\Psi_k^\top(\cdot)x_0$. Then $x(\cdot)$ solves the following SDE:
\begin{equation*}\label{wang4}
\left\{
\begin{aligned}
 dx(t)&=f_z^\top (t)x(t) dW(t),\q   t\in [0,T), \\
 \displaystyle
 x(0)&=x_0.
\end{aligned}
\right.
\end{equation*}
Then, by Lemma \ref{aaaa},
$$\me\big(\sup_{0\leq t\leq T}|x(t)|^p\big)\leq C|x_0|^p.$$
Consequently,
$$\me\big(\sup_{0\leq t\leq T}|\Psi(t)|^p\big)=\sup_{x_0\in\dbR^n}\frac{\me\big(\sup_{0\leq t\leq T}|x(t)|^p\big)}{|x_0|^p}\leq C.$$
Here $C$ depends only on $p,\,L$ and $T$. Similarly, one can prove $\ds \me\big(\sup_{0\leq t\leq T}|\Phi_k(t)|^p\big)\leq C$, and then \eqref{wang66} is proved.\\

Next, we only prove the second part $\ds \me\big(\sup_{s\leq t\leq T}|\Phi_k(t)\Psi_k(s)|^p\big)\leq C$ of \eqref{wang41}. The first one can be proved with the similar procedure.
For any $x_0\in \dbR^n$, set $\ds x_s(t)=\Phi_k(t)\Psi_k(s)x_0$. Then $x_s(t)$ solves the following SDE:
\begin{equation*}\label{wang42}
\left\{
\begin{split}
 dx_s(t)&=(f_z)^2 x_s(t) dt-f_zx_s(t) dW(t),\q   t\in [s,T), \\
 \displaystyle
 x_s(s)&=x_0.
\end{split}
\right.
\end{equation*}
Then, also by Lemma \ref{aaaa},
$$\me\big(\sup_{s\leq t\leq T}|\Phi_k(t)\Psi_k(s)x_0|^p\big)=\me\big(\sup_{s\leq t\leq T}|x_s(t)|^p\big)\leq C|x_0|^p,$$
where $C$ depends only on $p,\,L$ and $T$.

\ms

Now, by Eq. \eqref{wang2}, one has
\begin{equation}\label{bai12}
\begin{aligned}
&\me|(\Phi_k(t)-\Phi_k(s))\Psi_k(T_0)|^p\\
=&\me\Big|\int_s^t (f_z(\tau))^2\Phi_k(\tau)  d\tau\Psi_k(T_0)
    +\int_s^t f_z(\tau)\Phi_k(\tau)  dW(\tau)\Psi_k(T_0)\Big|^p\\
\leq&C\me\Big(\int_s^t\big|\Phi_k(\tau)\Psi_k(T_0)\big| d\tau\Big)^p+C\me\Big|\int_s^t f_z(\tau)\Phi_k(\tau)  dW(\tau)\Psi_k(T_0)\Big|^p\\
:=&C J_1+CJ_2.
\end{aligned}
\end{equation}
For $J_1$, by \eqref{wang41}, we have
\begin{equation}\label{bai43}
\begin{aligned}
J_1\leq& \me\int_s^t|\Phi_k(\tau)\Psi_k(T_0)|^p d\tau\Big(\int_s^t1 d\tau\Big)^{p-1}\\
=&\int_s^t\me|\Phi_k(\tau)\Psi_k(T_0)|^p d\tau(t-s)^{p-1}
\leq C(t-s)^p,
\end{aligned}
\end{equation}
where $C$ depends only on $p,\,L$ and $T$.
For $J_2$, by \eqref{wang41}, H\"older's inequality and Burkholder-Davis-Gundy inequality, we also can obtain
\begin{equation}\label{bai44}
\begin{aligned}
J_2=&\me\Big|\int_s^t f_z(\tau)\Phi_k(\tau)  dW(\tau)\Psi_k(T_0)\Big|^p
            =\me\Big|\int_s^t f_z(\tau)\Phi_k(\tau)\Psi_k(s)  dW(\tau)\Phi_k(s)\Psi_k(T_0)\Big|^p\\
\leq & \Big(\me\Big|\int_s^t f_z(\tau)\Phi_k(\tau)\Psi_k(s) dW(\tau)\Big|^{2p}\Big)^{1/2}\Big(\me|\Phi_k(s)\Psi_k(T_0)|^{2p}\Big)^{1/2}\\
\leq &C\Big(\me\Big(\int_s^t|\Phi_k(\tau)\Psi_k(s)|^2 d\tau\Big)^{p}\Big)^{1/2}\\
\leq& C\bigg\{\me\Big[\Big(\int_s^t1 d\tau\Big)^{\frac{p-1}{p}}\Big(\int_s^t|\Phi_k(\tau)\Psi_k(s)|^{2p} d\tau\Big)^{\frac{1}{p}}\Big]^p\bigg\}^{1/2}\\
\leq & C (t-s)^{\frac{p}{2}},
\end{aligned}
\end{equation}
where $C$ depends only on $p,\,L$.
Combining \eqref{bai12}--\eqref{bai44}, we have \eqref{wang45}.

\ms

Finally, we prove \eqref{wang45a}. Indeed,
For any $0\leq \theta,\,s\leq t\leq T$,
$D_{\theta}(\Phi_k(s)\Psi_k(\cdot))$ satisfies the following SDE:
\begin{equation*}\label{wang12}
\left\{
\begin{split}
 dD_{\theta}(\Phi_k(s)\Psi_k(t))&=\Big(D_{\theta}(\Phi_k(s)\Psi_k(t))f_z(t_k,t,x(t),Z^k(t))\\
           &\q+(\Phi_k(s)\Psi_k(t))(f_{zx}D_{\theta}x(t)+f_{zz}D_{\theta}Z^k(t)) \Big)dW(t),\q   \theta\leq t\leq T, \\
 \displaystyle
 D_{\theta}(\Phi_k(s)\Psi_k(\theta))&=0,\\
 D_{\theta}(\Phi_k(s)\Psi_k(t))&=0,\q 0\leq t <\theta.
\end{split}
\right.
\end{equation*}
For any $x_0\in \dbR^n$, set $x_{\theta,s}(\cdot)=D_{\theta}(\Psi_k^\top (\cdot)\Phi_k^\top (s))x_0$ and $y_s(\cdot)=\Psi_k^\top (\cdot)\Phi_k^\top (s)x_0$. Then $x_{\theta,s}(\cdot)$ satisfies the following SDE:
\begin{equation*}\label{wang13}
\left\{
\begin{split}
 dx_{\theta,s}(t)&=\big(f_z^\top (t)x_{\theta,s}(t)
           +\big(D_{\theta}x^\top (t)f_{zx}^\top (t)+D_{\theta}{Z^k}^\top (t)f_{zz}^\top (t)\big)y_s(t)\big) dW(t),\q   \theta\leq t\leq T, \\
 \displaystyle
 x_{\theta,s}(\theta)&=0,\\
 x_{\theta,s}(t)&=0,\q 0\leq t <\theta\leq T.
\end{split}
\right.
\end{equation*}
For any $p\in [2, \,2p_0)$, by Lemma \ref{aaaa}, we have
\begin{equation}\label{bai31}
\begin{aligned}
&\me\big(\sup_{\theta\leq t\leq T}|D_{\theta}(\Psi_k^\top (t)\Phi_k^\top (s))x_0|^p\big)\\
\leq & C\me\Big(\int_\theta^T\big|\big(D_{\theta} x^\top (t)f_{zx}^\top (t)+D_{\theta}{Z^k}^\top (t)f_{zz}^\top (t)\big)y_s(t)\big|^2\mathrm dt\Big)^{\frac {p} 2}\\
\leq& C\me\Big(\int_\th^T|D_\th x(t)|^2|y_s(t)|^2dt\Big)^{\frac p 2}\\
      &+C  \bigg\{\me\Big(\int_\theta^T|D_\theta Z^k(t)|^2 dt\Big)^{\frac{2p_0}{2}}\bigg\}^{\frac{p}{2p_0}}\Big(\me\big(\sup_{s\leq t\leq T}|y_s(t)|^{\frac{2pp_0}{2p_0-p}}\big)\Big)^{\frac{2p_0-p}{2p_0}},\\
\end{aligned}
\end{equation}
and 
\begin{equation}\label{bai031}
\begin{aligned}
    \me\Big(\int_\th^T|D_\th x(t)|^2|y_s(t)|^2dt\Big)^{\frac p 2}
\leq&  \bigg\{\me\bigg[\int_\th^T|D_\th x(t)|^{p_0}dt
   \Big(\int_\th^T|y_s(t)|^{\frac{2p_0}{p_0-2}}dt\Big)^{\frac {p_0-2}{2}}\bigg]\bigg\}^{\frac p {p_0}}\\
\leq& C\bigg[\me\int_\th^T|D_\th x(t)|^{2p_0}dt\bigg]^{\frac{p}{2p_0}}
   \Big[\me\big(\sup_{\th\leq t\leq T}|y_s(t)|^{2p_0}\big)\Big]^{\frac{p}{2p_0}}\\
\leq& C \Big(\sup_{\theta\leq t\leq T}\me|D_{\theta} x(t)|^{2p_0}\Big)^{\frac{p}{2p_0}} 
  \Big[\me\big(\sup_{\th\leq t\leq T}|y_s(t)|^{2p_0}\big)\Big]^{\frac{p}{2p_0}},\\
\end{aligned}
\end{equation}
\begin{equation}\label{bai0031}
\begin{aligned}
 \me\Big(\int_\theta^T|D_\theta Z^k(t)|^2\mathrm dt\Big)^{\frac{2p_0}{2}}
 \leq  C\bigg\{\me|D_\th x(T)|^{2p)}+\me\Big(\int_\th^T|D_\th x(t)|dt\Big)^{2p_0}\bigg\},
\end{aligned}
\end{equation}
where $C$ depends only on $p_0$, $L$ and $T$.  \eqref{bai31}, together with \eqref{bai031}, \eqref{bai0031}  and \eqref{wang41}, yields \eqref{wang45a}. That completes the proof.
\end{proof}

\ms

Now, we can prove Lemma \ref{regzkb}.
\begin{proof}[\bf {Proof of Lemma~\ref{regzkb}}]
We split the proof into two steps.

\textbf{Step 1.}
Similar to \eqref{april7a}, we also obtain:
for any $k=0,1,\cdots,N-1$,
$k\leq j\leq N-1$ and $s\in [t_j,t_{j+1}]$,
\begin{equation}\label{april7b}
\begin{aligned}
&Z^k(s)-Z^k(t_j)=D_sY^k(s)-D_{t_j}Y^k(t_j)\\
=&\big(D_sY^k(s)-D_{t_j}Y^k(s)\big)+\big(D_{t_j}Y^k(s)-D_{t_j}Y^k(t_j)\big),
\end{aligned}
\end{equation}
and 
\begin{equation}\label{april13b}
\begin{aligned}
&\me|D_{s}Y^k(s)-D_{t_j}Y^k(s)|^2
\leq  K|\pi|.\\
\end{aligned}
\end{equation}

\ms

\textbf{Step 2.} We claim that, for any $s\in [t_j, t_{j+1}]$, there exists a constant $K$, such that
\begin{equation}\label{april14b}
\begin{aligned}
\me|D_{t_j}Y^k(s)-D_{t_j}Y^k(t_j)|\leq K|s-t_j|.
\end{aligned}
\end{equation}
In order to do this,
for any $\th\leq t_j$, applying It\^o's formula to $\Psi_k(\cdot)D_\th Y^k(\cdot)$, we obtain
\begin{equation*}\label{wangbai2}
\begin{aligned}
\Psi_k(t)D_\th Y^k(t)
=&\Psi_k(T)D_\th Y^k(T)+\int_t^{T}\Psi_k(s) f_xD_\th x(s)ds\\
   &+\int_t^T\Psi_k(s)\big(f_zD_\th Y^k(s)+D_\th Z^k(s)\big)dW(s),\, t\in [t_j,t_j+1],\,j\geq k.\\
\end{aligned}
\end{equation*}
Since $\Psi_k(\cdot)\Phi_k(\cdot)=I_n$, one can get
\begin{equation*}\label{wangbai3}
\begin{aligned}
D_\th Y^k(t)=&\me\Big(\Phi_k(t)\Psi_k(T)D_\th Y^k(T)+\Phi_k(t)\int_t^T\Psi_k(s)f_xD_\th x(s)ds\big|\mf_t\Big).\\
\end{aligned}
\end{equation*}
Then
\begin{equation}\label{april17b}
\begin{aligned}
&D_\th Y^k(t)-D_\th Y^k(t_j)\\
=&\me\big(\Phi_k(t)\Psi_k(T)D_\th Y^k(T)|\mf_t\big)-\me\big(\Phi_k(t_j)\Psi_k(T)D_\th Y^k(T)|\mf_{t_j}\big)\\
           &+\me\Big(\Phi_k(t)\int_t^T\Psi_k(s) f_xD_\th x(s)ds|\mf_t\Big)
               -\me\Big(\Phi_k(t_j)\int_{t_j}^T \Psi_k(s)f_xD_\th x(s)ds|\mf_{t_j}\Big)\\
:=&I_1+I_2.\\
\end{aligned}
\end{equation}

\ms

Now, we estimate $I_1$ and $I_2$, respectively.
$I_1$ can be written as 
\begin{equation}\label{april19b}
\begin{aligned}
I_1=&\me\Big((\Phi_k(t)-\Phi(t_j))\Psi_k(T)D_\th Y^k(T)\big|\mf_t\Big)\\
       &+\me(\Phi(t_j)\Psi_k(T)D_\th Y^k(T)|\mf_t)-\me(\Phi(t_j)\Psi_k(T)D_\th Y^k(T)|\mf_{t_j})\\
   :=&I_{11}+I_{12}.    
\end{aligned}
\end{equation}
By \eqref{wang45}, a direct calculate leads to
\begin{equation}\label{april190b}
\begin{aligned}
\me|I_{11}|^2\leq &\Big(\me|(\Phi_k(t)-\Phi(t_j))\Psi_k(T)|^4\me|D_\th Y^k(T)|^4\Big)^{1/2}\leq C|t-t_j|.
\end{aligned}
\end{equation}
Meanwhile, by Clark-Ocone representation formula,
\begin{equation*}\label{april191b}
\begin{aligned}
 &\Phi(t_j)\Psi_k(T)D_\th Y^k(T)=\me\big(\Phi(t_j)\Psi_k(T)D_\th Y^k(T)\big)+\int_0^T u_\th(s)dW(s),
 \end{aligned}
\end{equation*}
where $\ds u_\th(\cdot)=\me\Big(D_{\cdot}(\Phi(t_j)\Psi_k(T))D_\th Y^k(T)+\Phi(t_j)\Psi_k(T)D_{\cdot}D_\th Y^k(T)\big|\mf_{\cdot}\Big)$.
Therefore, by Lemma \ref{wang65}, one gets
\begin{equation*}\label{april191b}
\begin{aligned}
\me|u_\th(s)|^2
\leq &C\bigg\{\Big(\me| D_s(\Phi(t_j)\Psi_k(T))|^4\me| D_\th Y^k(T)|^4\Big)^{1/2} 
                              +\Big(\me| \Phi_k(t_j)\Psi_k(T)|^4\me| D_sD_\th Y^k(T)|^4\Big)^{1/2}\bigg\}\\
   \leq&C<\infty.
\end{aligned}
\end{equation*}
Thus, 
\begin{equation}\label{april192b}
\begin{aligned}
\me|I_{12}|^2=  \me\Big|\int_{t_j}^t u_\th(s)dW(s)\Big|^2=\me\int_{t_j}^t |u_\th(s)|^2ds\leq C|t-t_j|.
\end{aligned}
\end{equation}

\ms

For $I_2$, we can rewrite it as follows:
\begin{equation}\label{april18b}
\begin{aligned}
I_2=&\me\Big(\int_t^T(\Phi_k(t)-\Phi_k(t_j)) \Psi_k(s)f_xD_\th x(s)ds\big|\mf_t\Big)\\
       &+\me\Big(\Phi_k(t_j)\Big(\int_{t}^T \Psi_k(s)f_xD_\th x(s)ds-\int_{t_j}^T \Psi_k(s)f_xD_\th x(s)ds\Big)\big|\mf_t\Big)\\
       &+\me\Big(\int_{t_j}^T\Phi_k(t_j)\Psi_k(s) f_xD_\th x(s)ds|\mf_t\Big)
             -\me\Big(\int_{t_j}^T\Phi_k(t_j)\Psi_k(s) f_xD_\th x(s)ds|\mf_{t_j}\Big)\\
     :=&I_{21}+I_{22}+I_{23}.  
\end{aligned}
\end{equation}
It is easy to check that
\begin{equation}\label{april21b}
\begin{aligned}
&\me I_{21}^2\leq \Big(\me\int_t^T|(\Phi_k(t)-\Phi_k(t_j)) \Psi_k(s) |^4ds\Big)^{1/2}
          \Big(\me\int_t^T|f_xD_\th x(s) |^4ds\Big)^{1/2}\\
   \leq & K(t-t_j)\Big(\sup_{0\leq s\leq T}\me |D_\th x(s)|^4\Big)^{1/2} \\         
   \leq &K|t-t_j|,
\end{aligned}
\end{equation}
and
\begin{equation}\label{april210b}
\begin{aligned}
&\me I_{22}^2\leq |t-t_j| \me\int_{t_j}^t|\Phi_k(t_j) \Psi_k(s)|^2 |f_xD_\th x(s)|^2ds\\
   \leq & K(t-t_j)\Big(\me\int_{t_j}^t|\Phi_k(t_j) \Psi_k(s) |^4ds\Big)^{1/2}
          \Big(\me\int_{t_j}^t|f_xD_\th x(s) |^4ds\Big)^{1/2} \\         
   \leq &K|t-t_j|.
\end{aligned}
\end{equation}
Now, we are in the step to estimate $I_{23}$.
By Clark-Ocone representation formula,
$$
\int_{t_j}^T\Phi_k(t_j) \Psi_k(s) f_xD_\th x(s)ds=\me\int_{t_j}^T\Phi_k(t_j) \Psi_k(s) f_xD_\th x(s)ds
+\int_0^T v_\th (u)dW(u),
$$
where
\begin{equation}\label{april211b}
\begin{aligned}
&v_\th(u)\\
=&\me\Big( D_u\int_{t_j}^T \Phi_k(t_j)\Psi_k(s)f_xD_\th x(s)ds      \big|\mf_u\Big)\\
=&\me\Big(\int_{t_j}^T \Phi_k(t_j)D_u\Psi_k(s)f_xD_\th x(s)ds \big|\mf_u\Big)
      +\me\Big(\int_{t_j}^T\Phi_k(t_j)\Psi_k(s)D_u(f_xD_\th x(s))ds \big|\mf_u\Big)\\
:=&V_1+V_2.
\end{aligned}
\end{equation}

For $V_1$, by \eqref{wang45a}, it is easy to check that
\begin{equation}\label{april212b}
\begin{aligned}
\me|V_1|^2=\Big(\me\int_{t_j}^T |\Phi_k(t_j)D_u\Psi_k(s)|^4ds \me\int_{t_j}^T | f_xD_\th x(s)|^4ds\Big)^{1/2}\leq K.
\end{aligned}
\end{equation}

For $V_2$,
\begin{equation}\label{april213b}
\begin{aligned}
\me|V_2|^2
\leq &\me\Big|\int_{t_j}^T\Phi_k(t_j)\Psi_k(s)D_u(f_xD_\th x(s))ds\Big|^2\\
\leq &K \me\Big(\int_{t_j}^T|\Phi_k(t_j)\Psi_k(s)|\times \big( |D_\th x(s)D_u x(s)|\\
        &    +|D_\th x(s)D_u Z^{k}(s)|+|D_uD_\th x(s)| \big)ds\Big)^2.\\
\end{aligned}
\end{equation}
We estimate the right side of \eqref{april213b} term by term. By Lemma \ref{sviex} and H\"{o}lder's  inequality,
\begin{equation}\label{april214b}
\begin{aligned}
& \me\Big(\int_{t_j}^T|\Phi_k(t_j)\Psi_k(s)| |D_\th x(s)D_u x(s)| ds\Big)^2\\
\leq &\me\bigg\{\sup_{0\leq s\leq T} |\Phi_k(t_j)\Psi_k(s)|^2
               \Big(\int_{t_j}^T \big(|D_\th x(s)|^2 +|D_u x(s)|^2\big)ds\Big)^2\bigg\}\\
\leq &K \Big(\me\sup_{0\leq s\leq T} |\Phi_k(t_j)\Psi_k(s)|^{\frac{2p_0}{p_0-1}}\Big)^{\frac{p_0-1}{p_0}}
              \Big( \me\int_{t_j}^T \big(|D_\th x(s)|^{2p_0} +|D_u x(s)|^{2p_0}\big)ds\Big)^{\frac{1}{p_0}}\\    
\leq &K<\infty.    
\end{aligned}
\end{equation}
Similarly,
\begin{equation}\label{april217b}
\begin{aligned}
& \me\Big(\int_{t_j}^T|\Phi_k(t_j)\Psi_k(s)| |D_uD_\th x(s)| ds\Big)^2\\
\leq &\me\bigg\{\sup_{0\leq s\leq T} |\Phi_k(t_j)\Psi_k(s)|^2
               \Big(\int_{t_j}^T |D_uD_\th x(s)|ds\Big)^2\bigg\}\\
\leq &\Big(\me\sup_{0\leq s\leq T} |\Phi_k(t_j)\Psi_k(s)|^{\frac{2p_0}{p_0-2}}\Big)^{\frac{p_0-2}{p_0}}
              \Big( \me\int_{t_j}^T |D_uD_\th x(s)|^{p_0}ds\Big)^{\frac{2}{p_0}}\\    
\leq &K<\infty.    
\end{aligned}
\end{equation}
Now, we estimate the left terms in the right side of \eqref{april213b}. For any $k=0,1,\cdots,N-1$, 
applying Lemma \ref{aaaa}, one can get
\begin{equation*}\label{zk1aa}
\begin{aligned}
\me\Big(\int_{t_k}^T|D_\th Z^k(s)|^2ds\Big)^{p_0}
\leq  K\bigg\{\me|D_\th Y^k(T)|^{2p_0}+ \me\int_t^{t_{j+1}}|f_xD_\th x(s)|^{2p_0}ds\bigg\}
\leq  K<\infty.
\end{aligned}
\end{equation*}
Therefore,
\begin{equation}\label{april216b}
\begin{aligned}
& \me\Big(\int_{t_j}^T|\Phi_k(t_j)\Psi_k(s)| |D_\th x(s)D_u Z^{k}(s)| ds\Big)^2\\
\leq &\me\bigg\{\sup_{0\leq s\leq T} |\Phi_k(t_j)\Psi_k(s)|^2\Big( \sup_{0\leq s\leq T}|D_\th x(s)|^4
              + \Big(\int_{t_j}^T |D_u Z^{k}(s)|^2ds\Big)^2\Big)\bigg\}\\
\end{aligned}
\end{equation}
\begin{equation*}
\begin{aligned}
\leq &\Big(\me\sup_{0\leq s\leq T} |\Phi_k(t_j)\Psi_k(s)|^{\frac{2p_0}{p_0-2}}\Big)^{\frac{p_0-2}{p_0}} \\
         &\times  \Bigg\{\Big(\me\sup_{0\leq s\leq T} |D_\th x(s)|^{2p_0}\Big)^{\frac{2}{p_0}}
              +\bigg[ \me\Big(\int_{t_j}^T |D_u Z^k(s)|^2ds\Big)^{p_0}\bigg]^{\frac{2}{p_0}}\Bigg\}\\    
\leq &K<\infty.
\end{aligned}
\end{equation*}
Hence, \eqref{april211b}, together with \eqref{april212b}--\eqref{april216b}, yields that
\begin{equation}\label{april218b}
\begin{aligned}
\me|I_{23}|^2=\me\Big|\int_{t_j}^t v_\th (u)dW(u)\Big|^2
\leq \me\int_{t_j}^t |v_\th (u)|^2du\leq K|t-t_j|.
\end{aligned}
\end{equation}

Finally, by \eqref{april17b}--\eqref{april210b} and  \eqref{april218b},
one gets
$$\me|D_\th Y_k(t)-D_\th Y_k(t_j)|^2\leq K|t-t_j|,$$
which deduces \eqref{april14b} by setting $\th=t_j$. 
Now combining \eqref{april7b} with \eqref{april13b} and \eqref{april14b}, we 
have the regularity of $Z$ \eqref{april06b}.
\end{proof}

\begin{remark}
From the proof of Lemma \ref{regzka} and Lemma \ref{regzkb}, we can see that
when $f=f(t,s,x,y)$ in BSDE \eqref{ass3}, we only need $p_0=2$ in  assumption {\rm(A4)};
but when $f=f(t,s,x,z)$, $p_0>2$ 
is needed.
\end{remark}

\ms

\subsection{Proof of Theorem \ref{convergence}}
In this part, we prove our main result Theorem \ref{convergence}. Firstly, we need the following lemma on conditional expectation. 
One can refer to \cite{Wang15} for proof.

\begin{lemma}\label{conditional}
For any $\varphi(\cdot)\in L^2_{\mathbb{F}}(\O\times (0,T);\dbR^n)$ and $0\leq s<t \leq T$, write
$$\varphi_0=\frac{1}{t-s}\me\bigg(\int_s^t\varphi(\tau) d\tau\Big|\mf_s\bigg).$$
Then for any $\xi\in L^2_{\mf_s}(\O;\dbR^n)$, it holds that
$$\me\int_s^t|\varphi(\tau)-\varphi_0|^2 d\tau\leq \me\int_s^t|\varphi(\tau)-\xi|^2 d\tau.$$
\end{lemma}

The following lemma is on the relation
between $(Y^{\pi(\cdot)}(\cdot),Z^{\pi(\cdot)}(\cdot))$ and $(Y^{\pi(\cdot),\pi}(\cdot),Z^{\pi(\cdot),\pi}(\cdot))$.
\begin{lemma}\label{wang1}
Let {\rm (A1)--(A4)} hold. Then, for any $k=0,1,\cdots,N-1$,
\begin{equation}\label{kk011}
\begin{aligned}
   &\sup_{k\leq j\leq N}\me|Y^k(t_j)-Y^{k,\pi}(t_j)|^2+\me\int_{t_k}^T|Z^k(s)-Z^{k,\pi}(\t(s))|^2ds \leq   K|\pi|,
\end{aligned}
\end{equation}
where $K$ is a constant depending only on $L$ and $T$.
\end{lemma}

\begin{proof}
We split the proof into three steps.

\textbf{Step 1.}
For any $k=0,1,\cdots,N-1$ and $k\leq j\leq N-1$, denote
\begin{equation}\label{Ikj}
\begin{aligned}
&I_{k,j}=\sup_{t_j\leq t\leq t_{j+1}}\me|Y^k(t)-Y^{k,\pi}(t)|^2
               +\frac 1 2\me\int_{t_j}^{t_{j+1}}|Z^k(s)-\widehat Z^{k,\pi}(s)|^2ds;\\
&I_{k,N}=\me|g(t_k,x(T))-g(t_k,x^\pi(T))|^2.
\end{aligned}
\end{equation}

By Eq. \eqref{ass3} and \eqref{ass4}, for $j\geq k$, we have
\begin{equation*}\label{kj1}
\begin{aligned}
   &\big(Y^k(t_j)-Y^{k,\pi}(t_j)\big)+\int_{t_j}^{t_{j+1}}\big(Z^k(s)-\widehat Z^{k,\pi}(s)\big)dW(s)\\
=&\big(Y^k(t_{j+1})-Y^{k,\pi}(t_{j+1})\big)\\
  &      +\int_{t_j}^{t_{j+1}}\big(f(t_{k},s,x(s),Y^j(s),Z^k(s))
        -f(t_{k},t_j,x^\pi(t_j),Y^{j,\pi}(t_{j+1}),Z^{k,\pi}_0(t_{j}))\big)ds.\\
\end{aligned}
\end{equation*}
Squaring and Taking expectation on both sides of the above equation, we obtain
\begin{equation}\label{kj2}
\begin{aligned}
    &\me|Y^k(t_j)-Y^{k,\pi}(t_j)|^2+\me\int_{t_j}^{t_{j+1}}|Z^k(s)-\widehat Z^{k,\pi}(s)|^2ds\\
\leq & \Big(1+\frac{8\D_j}{\e}\Big)\me|Y^k(t_{j+1})-Y^{k,\pi}(t_{j+1})|^2
          \\
      & +\Big(8+\frac{\e}{\D_j}\Big)L^2\bigg\{\me\Big|\int_{t_j}^{t_{j+1}}\sqrt{s-t_j}ds\Big|^2
           +\me\Big|\int_{t_j}^{t_{j+1}}x(s)-x(t_j)ds\Big|^2\\
      &  +\me\Big|\int_{t_j}^{t_{j+1}}x(t_j)-x^\pi(t_j)ds\Big|^2 
          +\me\Big|\int_{t_j}^{t_{j+1}}Y^j(s)-Y^j(t_{j+1})ds\Big|^2\\
     &+\me\Big|\int_{t_j}^{t_{j+1}}Y^j(t_{j+1})-Y^{j,\pi}(t_{j+1})ds\Big|^2
            +\me\Big|\int_{t_j}^{t_{j+1}}Z^k(s)-Z^k(t_{j})ds\Big|^2\\
     & +\me\Big|\int_{t_j}^{t_{j+1}}\frac{1}{\D_{j}}\me\Big(\int_{t_{j}}^{t_{j+1}}\big(Z^k(t_{j})-Z^k(\t)\big)d\t
            \big|\mf_{t_{j}}\Big)ds\Big|^2\\
       &+\me\Big|\int_{t_j}^{t_{j+1}}\frac{1}{\D_{j}}\me\Big(\int_{t_{j}}^{t_{j+1}}
                \big(Z^k(\t)-\widehat Z^{k,\pi}(\t)\big)d\t
            \big|\mf_{t_{j}}\Big)ds\Big|^2\bigg\}\\
\leq & \Big(1+\frac{8\D_j}{\e}\Big)\me|Y^k(t_{j+1})-Y^{k,\pi}(t_{j+1})|^2\\
        &+\Big(8+\frac{\e}{\D_j}\Big)L^2\bigg\{K|\pi|^3+|\pi|^2\me|Y^j(t_{j+1})-Y^{j,\pi}(t_{j+1})|^2\\
        & \qq\qq +|\pi|\me\int_{t_{j}}^{t_{j+1}}|Z^k(s)-\widehat Z^{k,\pi}(s)|^2ds\bigg\}.\\
\end{aligned}
\end{equation}
Now, choosing $\ds \e=\frac{1}{2L^2}$, then for $\ds|\pi|\leq \frac{1}{16L^2}$, one gets
\begin{equation}\label{kj02}
\begin{aligned}
    &\me|Y^k(t_j)-Y^{k,\pi}(t_j)|^2+\frac 1 2 \me\int_{t_j}^{t_{j+1}}|Z^k(s)-\widehat Z^{k,\pi}(s)|^2ds\\
\leq & (1+16L^2|\pi|)\me|Y^k(t_{j+1})-Y^{k,\pi}(t_{j+1})|^2\\
        &+\Big(8|\pi|+\frac{1}{2L^2}\Big)L^2\bigg\{K|\pi|^2+|\pi|\me|Y^j(t_{j+1})-Y^{j,\pi}(t_{j+1})|^2\bigg\}.\\
\end{aligned}
\end{equation}

In the above inequality, we use 
Lemma \ref{sviex}, Theorem \ref{svie-error} and Lemma \ref{regyk}.
For simplicity, denote $\ds b=1+16L^2|\pi|,\, c=b|\pi|.$ Then,  by induction, we can get
\begin{equation}\label{kk4}
\begin{aligned}
I_{k,j}\leq & bI_{k,j+1}+cI_{j,j+1}+cK|\pi|\\
 \leq & b^{N-j}I_{k,N}+\sum_{l=0}^{N-j-1}b^{N-j-l-1}c(b+c)^l I_{j+l,N}+cK|\pi|\sum_{l=0}^{N-j-1}(b+c)^l.
\end{aligned}
\end{equation}

For any $k$, by (A2) and Theorem \ref{svie-error},
\begin{equation}\label{kk5}
\begin{aligned}
I_{k,N}=&\me|g(t_k,x(T))-g(t_k,x^\pi(T))|^2
\leq K\me|x(T)-x^\pi(T)|^2\leq K|\pi|.
\end{aligned}
\end{equation}
Also, it is easy to check that, 
\begin{equation}\label{kk6}
\begin{aligned}
&\sum_{l=0}^{N-j-1}b^{N-j-l-1}c(b+c)^l=b^{N-j-1}c\sum_{l=0}^{N-j-1}\frac{(b+c)^l}{b^l}
\leq  b^{N}\Big(1+\frac c b\Big)^{N}\leq e^{16L^2T}e^{T},\\
\end{aligned}
\end{equation}
and
\begin{equation}\label{kk8}
\begin{aligned}
&cK|\pi|\sum_{l=0}^{N-j-1}(b+c)^l= cK|\pi|\frac{(b+c)^{N-j}-1}{b+c-1}
=K b|\pi|^2 \frac{(b+b|\pi|)^{N-j}-1}{b+b|\pi|-1}.\\
\end{aligned}
\end{equation}
Since 
\begin{equation*}\label{kk9}
\begin{aligned}
  b+b|\pi|
=1+(1+16L^2)|\pi|+16L^2|\pi|^2
\leq1+(2+32L^2)|\pi|,
\end{aligned}
\end{equation*}
and $\ds (b+b|\pi|)-1\geq 16L^2|\pi|$, \eqref{kk8} turns into
\begin{equation}\label{kk10}
\begin{aligned}
&cK|\pi|\sum_{l=0}^{N-j-1}(b+c)^l\leq K (1+16L^2|\pi|)|\pi|^2 \frac{\big(1+(2+32L^2)|\pi|\big)^N}{16L^2|\pi|}\\
\leq &\frac K {16L^2} |\pi| (1+16L^2|\pi|) e^{(2+32L^2)T}\leq K|\pi|.
\end{aligned}
\end{equation}
Hence, \eqref{kk4}, together with \eqref{kk5}, \eqref{kk6} and \eqref{kk10}, yields that
\begin{equation}\label{kk11}
\begin{aligned}
   &\me|Y^k(t_j)-Y^{k,\pi}(t_j)|^2 \leq   K|\pi|.
\end{aligned}
\end{equation}
That is the first part of \eqref{kk011}.

\ms

\textbf{Step 2.}
Now, we estimate $\ds \me\int_{t_k}^T|Z^k(s)- Z^{k,\pi}(\t(s))|^2ds$. 
By \eqref{kj02},  summing from $j=k$ to $N-1$ leads to
\begin{equation*}\label{zkj1}
\begin{aligned}
    &\sum_{j=k}^{N-1}\me|Y^k(t_j)-Y^{k,\pi}(t_j)|^2+\frac 1 2\me\int_{t_k}^{t_{N}}|Z^k(s)-\widehat Z^{k,\pi}(s)|^2ds\\
\leq &(1+16L^2|\pi|)\sum_{j=k}^{N-1}\me|Y^k(t_{j+1})-Y^{k,\pi}(t_{j+1})|^2\\
       &+(1+16L^2|\pi|)|\pi|\sum_{j=k}^{N-1}\big\{K|\pi|+\me|Y^j(t_{j+1})-Y^{j,\pi}(t_{j+1})|^2\big\}.
\end{aligned}
\end{equation*}
Hence, by \eqref{kk11},
\begin{equation}\label{zkj2}
\begin{aligned}
    &\me\int_{t_k}^T|Z^k(s)-\widehat Z^{k,\pi}(s)|^2ds\\
\leq &32L^2|\pi|\sum_{j=k}^{N-1}\me|Y^k(t_{j+1})-Y^{k,\pi}(t_{j+1})|^2+2\me|Y^k(t_{N})-Y^{k,\pi}(t_{N})|^2\\
       &-2\me|Y^k(t_{k})-Y^{k,\pi}(t_{k})|^2+(2+32L^2|\pi|)|\pi|\sum_{j=k}^{N-1}\big\{K|\pi|
          +\me|Y^j(t_{j+1})-Y^{j,\pi}(t_{j+1})|^2\big\}\\
        \leq & K|\pi|.
\end{aligned}
\end{equation}

\ms

\textbf{Step 3.}
For any $k=0,1,\cdots,N-1$, and $k\leq j\leq N-1$, denote
$$\bar{Z}^k(t_j)=\frac{1}{\D_j}\me\Big(\int_{t_j}^{t_{j+1}}Z^k(s)ds|\mf_{t_j}\Big).$$
Then, by Lemma \ref{regzka}, Lemma \ref{regzkb} and Lemma \ref{conditional} and \eqref{zkj2}, 
a direct calculation leads to
\begin{equation*}\label{zkj3}
\begin{aligned}
    &\me\int_{t_k}^T|Z^k(s)-Z^{k,\pi}(\t(s))|^2ds=\sum_{j=k}^{N-1}\me\int_{t_j}^{t_{j+1}}|Z^k(s)-Z^{k,\pi}(\t(s))|^2ds\\
\leq &2\sum_{j=k}^{N-1}\me\int_{t_j}^{t_{j+1}}\big(|Z^k(s)-\bar{Z}^k(t_j)|^2+|\bar{Z}^k(t_j)-Z^{k,\pi}(t_j)|^2\big)ds\\
\leq &2\sum_{j=k}^{N-1}\me\int_{t_j}^{t_{j+1}}\Big(|Z^k(s)-Z^k(t_j)|^2+\Big|\frac{1}{\D_j}\me\Big(\int_{t_j}^{t_{j+1}}Z^k(u)-\widehat Z^{k,\pi}(u)du|\mf_{t_j}\Big)\Big|^2\Big)ds\\
\leq & K|\pi|+2\sum_{j=k}^{N-1}\me\int_{t_j}^{t_{j+1}}|Z^k(u)-\widehat Z^{k,\pi}(u)|^2du\\
\leq & K|\pi|.
\end{aligned}
\end{equation*}
That completes the proof.
\end{proof}

\ms

Now, we are in the step to prove Theorem \ref{convergence}.

\begin{proof}[\bf {Proof of Theorem~\ref{convergence}}]

By Lemma \ref{wang1}, we can see that $\ds \sup_{0\leq t\leq T}\me|Y(\t(t))-Y^{\pi(t),\pi}(\t(t))|^2\leq K|\pi|$ is true.

For the second term, $\ds\me\int_0^T\int_t^T|Z(t,s)-Z^{\pi(t),\pi}(\t(s))|^2ds$, on the left side of \eqref{conver},
It is easy to check that
\begin{equation}\label{conz1}
\begin{aligned}
&\me\int_0^T\int_t^T|Z(t,s)-Z^{\pi(t),\pi}(\t(s))|^2 dsdt\\
\leq &\me\int_0^T\int_t^T|(Z(t,s)-Z(\t(t),s))+(Z(\t(t),s)-\widehat Z^{\pi(t),\pi}(s))\\
       &\qq\qq\qq+(\widehat Z^{\pi(t),\pi}(s)-Z^{\pi(t),\pi}(\t(s)))|^2 dsdt\\
\eal
\ee
\begin{equation*}
\bal
\leq &3\me\int_0^T\int_t^T|Z(t,s)-Z(\t(t),s)|^2 dsdt
          +3\me\int_0^T\int_{\t(t)}^T|Z(\t(t),s)-\widehat Z^{\pi(t),\pi}(s)|^2 dsdt\\
       &\qq\qq\qq+3\me\int_0^T\int_{\t(t)}^T|\widehat Z^{\pi(t),\pi}(s)-Z^{\pi(t),\pi}(\t(s))|^2 dsdt.\\
\end{aligned}
\end{equation*}
By Lemma \ref{estimateyz}, one has
\begin{equation}\label{conz2}
\begin{aligned}
\me\int_0^T\int_t^T|Z(t,s)-Z(\t(t),s)|^2 dsdt\leq C\int_0^T(t-\t(t))dt\leq C|\pi|.
\end{aligned}
\end{equation}
For the third term on the right side of \eqref{conz1}, by \eqref{zkj2} and \eqref{zkj3}, one has
\begin{equation}\label{conz3}
\begin{aligned}
    &\me\int_0^T\int_{\t(t)}^T|\widehat Z^{\pi(t),\pi}(s)-Z^{\pi(t),\pi}(\t(s))|^2 dsdt\\
\leq &2\me\int_0^T\int_{\t(t)}^T|\widehat Z^{\pi(t),\pi}(s)-Z^{\pi(t)}(s)|^2 dsdt+2\me\int_0^T\int_{\t(t)}^T|Z^{\pi(t)}(s)-Z^{\pi(t),\pi}(\t(s))|^2 dsdt\\
\leq & K|\pi|.
\end{aligned}
\end{equation}

Now, we estimate $\ds \me\int_0^T\int_{\t(t)}^T|Z(\t(t),s)-\widehat Z^{\pi(t),\pi}(s)|^2 dsdt$.
By Eq. \eqref{bsvie1} and \eqref{ass4}, for any $k=0,1,\cdots,N-1$, one can easily calculate 
\begin{equation}\label{conz4}
\begin{aligned}
    &\me|Y(t_k)-Y^{k,\pi}(t_k)|^2+\me\int_{t_k}^T|Z(t_k,s)-\widehat Z^{k,\pi}(s)|^2 ds\\
\leq & \me\bigg| \big(g(t_k,x(T))-g(t_k,x^\pi(T))\big)\\
       & +\sum_{l=k}^{N-1}\int_{t_l}^{t_{l+1}}\big(f(t_k,s,x(s),Y(s),Z(t_k,s))
            -f(t_k,t_l,x^\pi(t_l),Y^{l,\pi}(t_{l+1}),\widehat Z^{k,\pi}_0(t_l))\big)ds\bigg|^2\\
\leq & 2L^2\me|x(T)-x^\pi(T)|^2\\
       &+2N|\pi|\sum_{l=k}^{N-1}\me\int_{t_l}^{t_{l+1}}|f(t_k,s,x(s),Y(s),Z(t_k,s))
               -f(t_k,t_l,x^\pi(t_l),Y^{l,\pi}(t_{l+1}),\widehat Z^{k,\pi}_0(t_l))|^2ds\\
\end{aligned}
\end{equation}
\begin{equation*}
\begin{aligned}
\leq & K|\pi|+K\sum_{l=k}^{N-1}\me\int_{t_l}^{t_{l+1}}\Big(|s-t_l|+|x(s)-x(t_l)|^2+|x(t_l)-x^\pi(t_l)|^2\\
       &\qq+|Y(s)-Y^l(s)|^2+|Y^l(s)-Y^l(t_{l+1})|^2+|Y^l(t_{l+1})-Y^{l,\pi}(t_{l+1})|^2\\
       &\qq+|Z(t_k,s)-Z^k(s)|^2+|Z^k(s)-Z^{k,\pi}(t_{l})|^2\Big)ds\\
\leq & K|\pi|.
\end{aligned}
\end{equation*}
Here, we use Theorem \ref{yong1}, Lemma \ref{regyk}, Lemma \ref{wang1}, 
Proposition \ref{yyk} and \eqref{zkj3}.
Now, \eqref{conz1}, together with  \eqref{conz2}--\eqref{conz4}, yields that
\begin{equation}\label{conz5}
\begin{aligned}
\me\int_0^T\int_t^T|Z(t,s)-Z^{\pi(t),\pi}(\t(s))|^2 dsdt\leq K|\pi|.
\end{aligned}
\end{equation}
That completes the proof of the  convergent 
speed of the Euler method for BSVIE \eqref{bsvie1}.
\end{proof}

\section{A Numerical example}

In this section, we mainly present a numerical example. 
Consider the following BSVIE:
\begin{equation}\label{ex1}
\begin{aligned}
Y(t)=t \sin(W(1))+\int_t^1\frac t 2 \sin(W(s))ds-\int_t^1Z(t,s)dW(s),\,\,t\in [0,1],
\end{aligned}
\end{equation}
with $T=d=n=1$,
which admits a unique solution $\big(t\sin (W(t)),\, t\cos (W(s))\big)$.

In Figure 1, choosing $N=100$ (i.e. $|\pi|=0.01$), we simulate true solution $(Y(t),Z(t,s))$ 
(in red) and its approximation $(Y^{\pi(\cdot),\pi}(\cdot),Z^{\pi(\cdot),\pi}(\cdot))$ (in blue). 
 For the $Z$ part, we take three cases: $t=0.1,\, 0.2,\, 0.3$ and one sample path $\o\in\O$.

\begin{figure}\label{figs1}
 \begin{center}
  \includegraphics[width=0.4\textwidth,height=0.2\textheight]{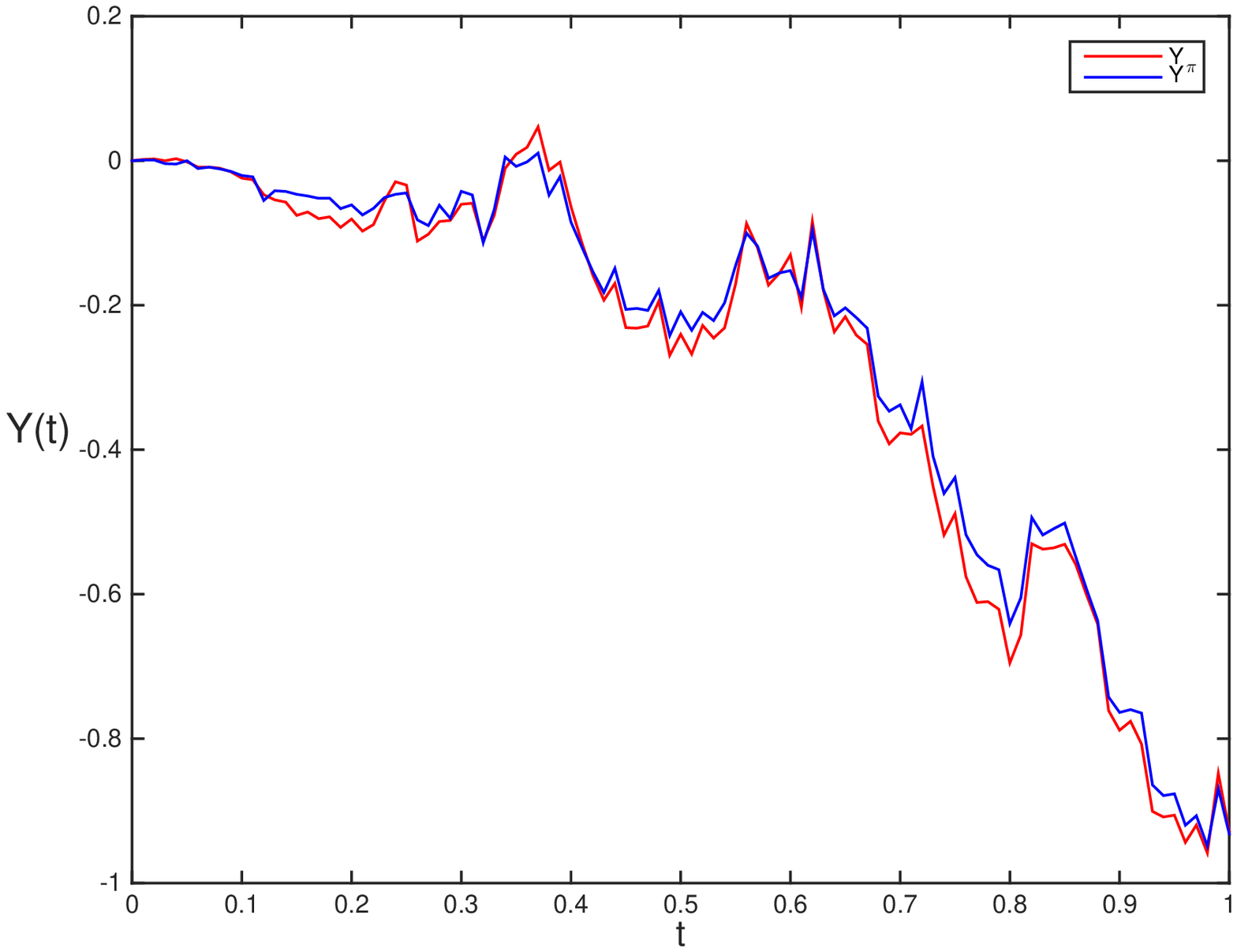}
\includegraphics[width=0.4\textwidth,height=0.2\textheight]{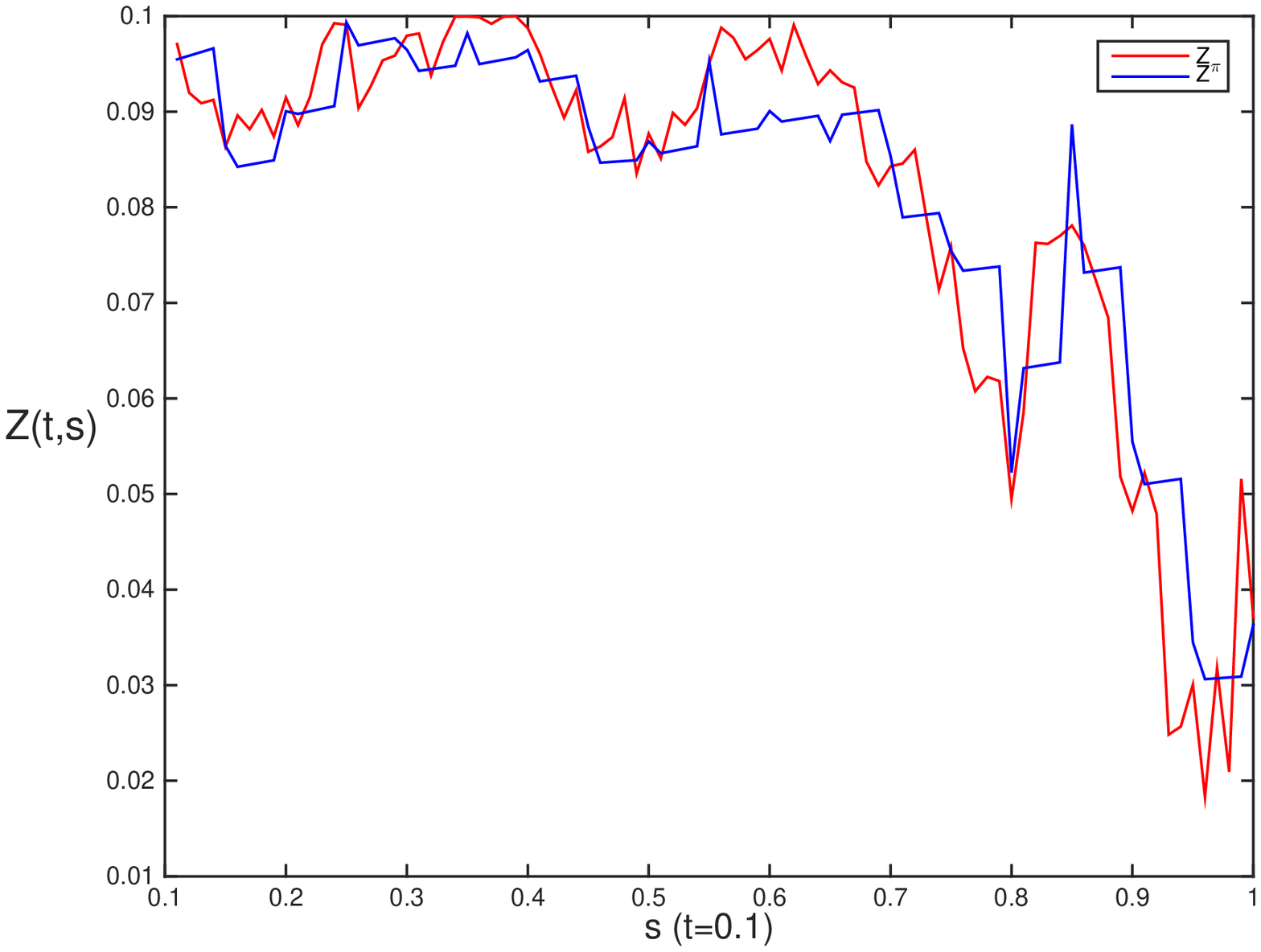}
 \includegraphics[width=0.4\textwidth,height=0.2\textheight]{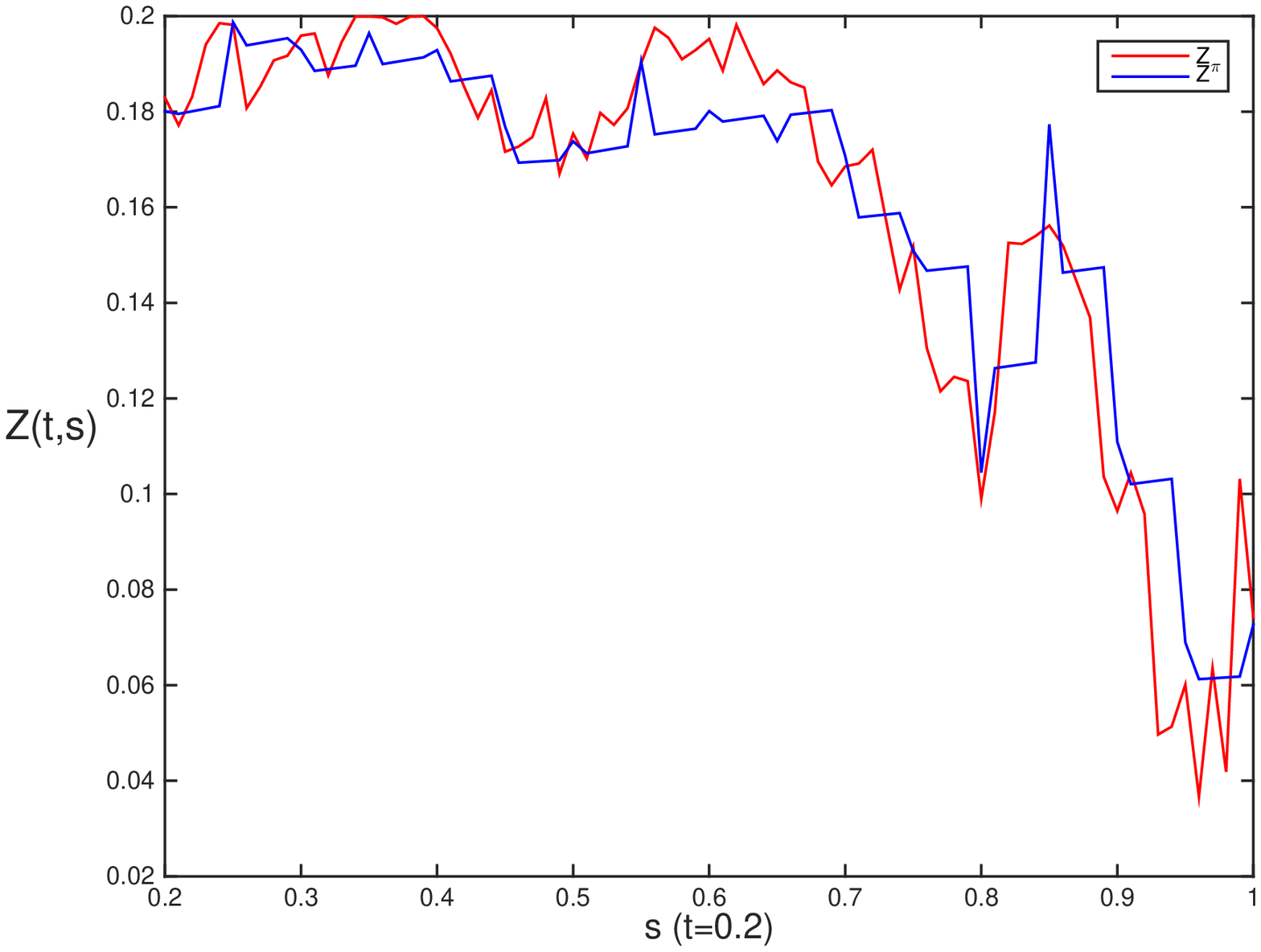}
\includegraphics[width=0.4\textwidth,height=0.2\textheight]{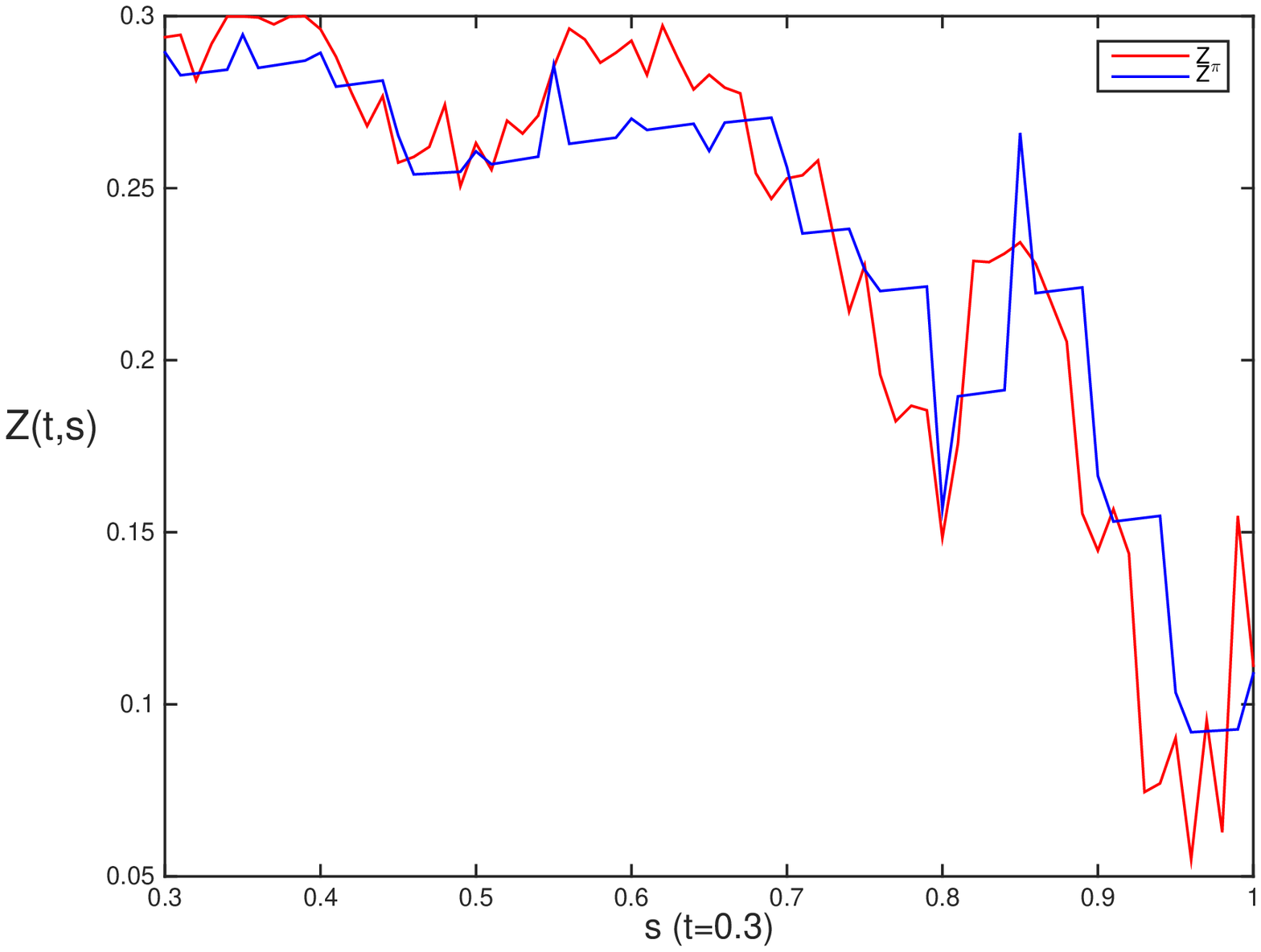}
 \caption{$(Y(t),Z(t,s))$ and its approximation $(Y^{\pi(t),\pi}(t),Z^{\pi(t),\pi}(\t(s)))$. 
 For  $Z(t,s)$ and its approximation $Z^{\pi(t),\pi}(\t(s))$, we choose 
 $\ds t=0.1,\, 0.2,\, 0.3$ for  one sample path $\o\in\O$.}
  \end{center}
  \end{figure}

\section*{Acknowledgement}

This work was carried out during the stay of the author at University of Central Florida, USA. The author
would like to thank the Department of Mathematics for its
hospitality, and the financial support from China Scholarship Council.  The author also gratefully acknowledges Professor Jiongmin Yong for stimulating discussions during this
work.


\end{document}